\documentclass{amsart}
\usepackage{amsthm,amsfonts,amsmath,amssymb}
\usepackage{geometry}
\usepackage[draft=true]{hyperref}
\usepackage{cite}
\usepackage{hyperref}
\usepackage[english]{babel}

\usepackage{epsfig}
\usepackage{graphicx}
\usepackage{colordvi}
\usepackage{graphics}
\usepackage{color}
\usepackage{float}
\usepackage{bm}
\usepackage{verbatim} 
\usepackage{epstopdf}

\numberwithin{equation}{section}

\newtheorem{theorem}{Theorem}[section]
\newtheorem{thm}{Theorem}[section]
\newtheorem{lem}{Lemma}[section]
\newtheorem{lemma}{Lemma}[section]
\newtheorem{dfn}[thm]{Definition}
\newtheorem{rem}[thm]{Remark}

\newtheorem{alg}[thm]{Algorithm}
\newtheorem{remark}[thm]{Remark}

\def \bX{{\mathbf X}}
\def \bV{{\mathbf V}}
\def \bH{{\mathbf H}}

\def \bu{{\mathbf u}}
\def \bv{{\mathbf v}}
\def \bV{{\mathbf V}}
\def \bw{{\mathbf w}}
\def \be{{\mathbf e}}

\def \bf{{\mathbf f}}
\def \b0{{\mathbf 0}}
\def \bx{{\mathbf x}}
\def \bP{{\mathbf P}}
\def \bL{{\mathbf L}}
\def \bk{{\mathbf k}}

\def \bphi{{\mathbf \phi}}
\def \bta{{\mathbf \eta}}

\usepackage{caption}
\usepackage{adjustbox}
\begin{document}
\title[On non-linear Leray model...]{Improving accuracy in the Leray model for incompressible non-isothermal flows via adaptive deconvolution-based nonlinear filtering}
\author{Mine AKBAS}
\address{Department of Mathematics, D\"{u}zce University, D\"{u}zce-Turkey}
\email{mineakbas@duzce.edu.tr}
\author{Abigail Bowers}
\address{Department of Mathematics, Florida Polytechnic University, Lakeland, FL 33805.}
\email{abowers@floridapoly.edu}
\keywords{\textbf{ }Non-isothermal fluid flows, finite element method, Leray regularization, non-linear filtering.}

\begin{abstract}
This paper considers a Leray regularization model of incompressible, non-isothermal fluid flows which uses nonlinear filtering based on indicator functions, and introduces an efficient numerical method for solving it. The proposed method uses a multi-step, second-order temporal discretization with a finite element (FE) spatial discretization in such a way that the resulting algorithm is linear at each time level, and decouples the evolution equations from the velocity filter step. Since the indicator function chosen in this model is mathematically based on approximation theory, the proposed numerical algorithm can be analyzed robustly, i.e the stability and convergence of the method is provable. A series of numerical tests are carried out to verify the theoretical convergence rates, and to compare the algorithm with direct numerical simulation and the usual Leray-$\alpha$ model of the flow problem.
\end{abstract}

\maketitle

\section{Introduction}
Recent work utilizing numerical methods/models with physical-phenomenology based indicator functions has been successful in yielding more accurate solutions to incompressible flow problems. This concept is a central feature of dynamic Smagorinsky models \cite{Vreman03, Vreman04, JK08, CK11, YC16}, adaptive regularization models \cite{BRTT12}, and adaptive filter-based stabilization methods \cite{LRT12, OX13}. These models incorporate an indicator function to identify the regions of the domain where stabilization is necessary. Many different types of indicator functions have been used, all of which (to our knowledge) relied upon physical phenomenology, and were not conducive to rigorous mathematical analysis, until \cite{BR13}. This work proposed a deconvolution-based indicator and proved that the convergence rate was improved due to its use.

We expand on the work of \cite{BR13} by studying a variation of the Leray-$\alpha$ model that uses deconvolution-based nonlinear filtering for incompressible, non-isothermal fluid flows. Due to its attractive properties, Leray-$\alpha$ models have been widely studied from both a mathematical and computational point of view \cite{BRTT12, GHMP08, GHMP11, GKT08, LKT08}. In particular, these models conserve energy  and $2D$-enstrophy \cite{CHOT05, R07}, and cascade energy through the inertial range at the same rate as the NSE, up to a filtering radius dependent wave number \cite{CHOT05}. Further, in the finite element context, the model discretizations can be easily used in existing legacy Navier-Stokes codes in such a way that filtering equations are decoupled from the system in an unconditionally stable way. In computations, this decoupling leads to no significant extra cost from using these models when compared with the usual discretization of the system. However, these methods have lower order accuracy for smooth flows, i.e., $\mathcal{O}(\alpha^2)$, which can lead to over-regularized solutions, can cause higher computational costs due to the requirement of a smaller $\alpha$, and as a consequence, a finer mesh \cite{GH03, GH06, LMNR08, BR12, LMNR082, LR12}. One remedy to avoid over regularized solutions and to improve numerical accuracy of Leray-$\alpha$ models is to use adaptive deconvolution-based nonlinear filtering in these models \cite{LMNR08, BRTT12, LRT12}.

The Boussinesq system is a coupled multiphysics flow problem which describes incompressible, non-isothermal flows, and is given by
\begin{eqnarray}\label{NSHeat}
\bu_t + \bu\cdot\nabla \bu -Re^{-1}\Delta \bu +\nabla p - Ri T\hat{\bk}&=& \bf, \\
\nabla\cdot \bu &=& 0 ,\\
T_t + \bu\cdot\nabla T-(RePr)^{-1}\Delta T &=& \gamma ,
\end{eqnarray}
where $\bu$ is the velocity, $p$ pressure, $T$ temperature, $\bf$ is a given force, $\hat{\bk}:= <0,\ldots,0,1>$ is the unit vector, and $Re$, $Ri$, and $Pr$ are the Reynolds, Richardson, and Prantdl numbers, respectively.

The Leray-Boussinesq model in $(0,t^*]\times \Omega$ with appropriate boundary and initial conditions is given by
\begin{eqnarray}
\bu_t + \overline{\bu}\cdot\nabla \bu -Re^{-1}\Delta \bu +\nabla p - RiT\hat{\bk}&=&\bf,\\
\nabla\cdot \bu &=& 0 ,\\
T_t + \bu\cdot\nabla T-(RePr)^{-1}\Delta T &=& \gamma,\\
-\alpha^2 \nabla \cdot ( a(\bu) \nabla \overline{\bu} ) + \overline{\bu} + \nabla \lambda & = & \bu,\label{leraybou1}\\
\nabla \cdot \overline{\bu} & = & 0.
\end{eqnarray}
where $\alpha>0$ is the spatial filtering radius, and $a(\cdot)(x)$ is a function satisfying
$
0<a(\cdot)(x)\leq 1 \label{indicator1}
$
and
\begin{eqnarray}
& a(\cdot)(x) \approx 0 \mbox{  in regions where $\phi$ does not need regularization}, \\
& a(\cdot)(x) \approx 1 \mbox{  in regions where $\phi$ does need regularization}. \label{indicator2}
\end{eqnarray} 

These functions used in \eqref{leraybou1} having properties \eqref{indicator1}-\eqref{indicator2} are known as indicator functions. We make a note here that if $a(\bu)=1$, then this filter coincides with that of the Leray-$\alpha$ model.

The indicator function used in this paper is an in \cite{BR13} and takes the form
\begin{equation}
a_{D_N}(\bu):=|\bu-D_N^h\widetilde{\bu}^h|,\label{indicator3}
\end{equation}
where $D_N^h$ denotes $N^{\text{th}}$ order van Cittert deconvolution operator with the discrete linear Helmholtz filter, and $\widetilde{\bu}^h$ the discrete Helmholtz filtering of $\bu$. At the continuous level, $ D_N(\widetilde{\bu})\approx \bu$ in regions where $\bu$ does not need a regularization, and so the function $a_{D_N}$ defined in \eqref{indicator3} satisfies \eqref{indicator2}. We emphasize that in numerical simulations, the Helmholtz (differential) filter, which is the solution of a Helmholtz equation, can be only approximated. In finite element implementations, one has to discretize the Helmholtz equation in the velocity space, which gives the so-called discrete Helmholtz (differential) filter.

The Leray regularization model with van Cittert approximation deconvolution-based indicator functions inherits desirable properties of the Leray-$\alpha$ models such as well-posedness, and energy conservation. In addition, this model leads to robust discretizations with accuracy increased from $\mathcal{O}(\alpha^2)$ to $\mathcal{O}(\alpha^{2N+2})$. This is due to the fact that van Cittert approxiamte deconvolution operators are well-established and mathematically grounded \cite{SA99, SAK01,DE06}. The application of the model for the incompressible NSE can be seen in \cite{BR13} where the method was able to be analyzed rigorously, i.e., the stability of the method and convergence to the NSE was proven. In addition, numerical results revealed that the method was very successful in getting in the fluid flow approximations on much coarse meshes. 

The aim of this paper is to extend this method from \cite{BR13} for the incompressible non-isothermal fluid flows, and develop an efficient finite element discretization for the model.  The application of the standard finite element method for incompressible non-isothermal fluid flows, which aims to simulate all scales, is reported in literature \cite{GLCS80, M87}. However, this system is a coupling of the Navier-Stokes equations to the transport equation, and thus is capable of generating both velocity and temperature scales sufficiently small to prohibit practical full resolution of solutions in many situations. Therefore, the method we propose here aims to efficiently truncate velocity scales, and the expectation is to get much more accurate solutions on coarser meshes.

This paper is arranged as follows. Section 2 gathers preliminary results for the finite element analysis. Section 3 defines the filtering and deconvolution operators, and gives some properties for them. Section 4 introduces a numerical method for the nonlinear Leray model via van Cittert approximate deconvolution-based indicator functions. The proposed method uses linearized BDF2 (BDF2LE) temporal and finite element spatial discretization in such a way that the evolution equations and the filter step are linear at each time step, and are decoupled from one another. We prove unconditional stability with respect to time step, and optimal convergence to the model both in time and space. Section 5 provides two numerical experiments; the first one verifies theoretical convergence rates, the second shows the effectiveness of the algorithm over BDF2LE-FE discretization and Leray-$\alpha$ model of the Boussinesq system.

\section{Notation and Preliminaries}

We consider the domain $\Omega\subset\mathbb{R}^d$ (d=2 or 3) to be a convex polygon or polyhedra. The $L^{2}(\Omega)$ norm and inner product will be denoted as $\|\cdot\|$ and $(\cdot,\cdot)$, the $H^{k}(\Omega)$ norm by $\|\cdot\|_{k}$, and the $L^{\infty}(\Omega)$ norm by $\|\cdot\|_{\infty}$. All other norms will be clearly labeled.\\
We will consider wall-bounded flows for our analysis, and the natural function spaces for this setting are 
\begin{eqnarray}
&& \bX := (H^1_{0}(\Omega))^d:=\{\bv\in (H^{1}(\Omega))^d, \bv=0 \text{ on } \partial\Omega\} , \\
&& Q := L^2_{0}(\Omega):=\{q\in L^{2}(\Omega), \int_{\Omega}q=0\}, \\
&& Y := H^1_0(\Omega).
\end{eqnarray}
Further, we define the space $\bV \subset \bX$ to be the divergence free subset of $\bX$. The dual space of $\bX$ is denoted by $\bH^{-1}$ with the norm
\[
\|\bf\|_{-1}:=\sup\limits_{0\neq \bv\in \bX}\frac{|(\bf, \bv)|}{\|\nabla \bv\|}.
\]
In the stability and convergence analysis we will frequently use the Poincar{\'{e}}-Friedrichs' Inequality: There exists a constant $C_P:=C_P(\Omega)$, which  depends only on the size of the domain, such that 
\[
\| \bv \| \le C_P \| \nabla \bv \|,\hspace{2mm}\forall \bv\in \bX.
\]
We use skew symmetrized trilinear forms for the non-linear terms to ensure stability of the numerical method:
\begin{eqnarray*}
b(\bu, \bv, \bw) &:=& \frac{1}{2}\left[(\bu\cdot\nabla \bv, \bw) - (\bu\cdot\nabla \bw, \bv)\right], \ \forall \bu, \bv, \bw \in \bX,\\
c(\bu,\theta,\psi) &:=& \frac{1}{2}\left[(\bu\cdot\nabla \theta,\psi) - (\bu\cdot\nabla \psi,\theta)\right], \,\,\,\,\,\, \forall \bu \in \bX \text{ and } \theta, \psi\in Y.
\end{eqnarray*}
There are several important estimates for these operators that we will employ in subsequent sections, which are proven in \cite{Layton08}. Analogous estimates hold for the $c$ operator.
\begin{lemma} \label{trilinearbound}
For $\bu, \bv, \bw \in \bX$, and also $\bv, \nabla\bv\in \bL^{\infty}(\Omega)$ for \eqref{tribound1}, the trilinear term $b(\bu, \bv, \bw)$ is bounded by
\begin{eqnarray}
b(\bu, \bv, \bw) &\leq& \frac{1}{2} \left( \|\bu\|\|\nabla \bv\|_{\infty}\|\bw\| + \|\bu\|\|\bv\|_{\infty}\|\nabla \bw\|\right), \label{tribound1}\\
b(\bu, \bv, \bw) &\le& C \|\nabla \bu\|\|\nabla \bv\|\|\nabla \bw\|. \label{tribound3}\\
b(\bu, \bv, \bw) &\le& C \|\bu\|^{1/2}\|\nabla \bu\|^{1/2}\|\nabla \bv\|\|\nabla \bw\|, \label{tribound2}
\end{eqnarray}
\end{lemma}
\begin{proof}
The first of these bounds can be proved by applying the generalized H\"{o}lder Inequality with $p=2,\, q=\infty, \, r=2$ to the definition of $b(\cdot, \cdot, \cdot)$. The second bound follows from the generalized H\"{o}lder Inequality with $p=2,\, q=4, \, r=2$, the Ladyzhenskaya Inequality together with the Poincar{\'{e}}-Friedrichs' Inequality, see \cite{Layton08}. The last bound can be obtained using similar tools.
\end{proof}
\noindent We also apply the Agmon's Inequality.
\begin{lemma}\label{Agmon}
Assume $\bv\in \bX\cap \bH_2(\Omega)$. Then it holds
\begin{eqnarray*}
\|\bv\|_{L^{\infty}} \leq  C \|\bv\|_{\bH^1}^{1/2}\|\bv\|_{H^2}^{1/2},\hspace{4mm}\text{d=3},\\
\|\bv\|_{L^{\infty}} \leq  C \|\bv\|^{1/2}\|\bv\|_{H^2}^{1/2},\hspace{4mm}\text{d=2}.
\end{eqnarray*}
\end{lemma}
\noindent We assume a regular, conforming mesh $\tau_h$, with maximum element diameter $h$, and associated 
velocity-pressure-temperature finite element (FE) spaces $\bX_h\subset \bX$, $Q_h\subset Q$, and $Y_h\subset Y$ satisfying approximation properties of piecewise polynomials of local degree $k, k-1$ and $k$, respectively, \cite{GR86,BS94,Zhang05,qin:phd}:
\begin{eqnarray}\label{apprpro}
    \inf_{\bv_h\in \bX_{h}}\left(\| \bu - \bv_h \| + h \|\nabla( \bu - \bv_h )\| \right) & \le & C h^{k+1} \| \bu \|_{k+1},\;\; \bu \in
    \bH^{k+1}(\Omega),\\
    \inf_{q_h \in Q_{h}} \| p - q_h \| & \le & C h^{k} \| p \|_{k},\;\;  \quad \quad p \in
    H^{k}(\Omega),\\
    \inf_{\omega_h\in Y_{h}}\left(\| T - \omega_h \| +  h \|\nabla( T - \omega_h )\| \right)& \le & C h^{k+1} \| T \|_{k+1},\;\; T \in
    H^{k+1}(\Omega).
\end{eqnarray}
We also assume further that the mesh is sufficiently regular such that the inverse inequality holds: $\forall \bv_h\in \bX_h$,
\[
\|\nabla \bv_h\|\leq C h^{-1}\|\bv_h\|.
\]
The finite element spaces for velocity-pressure are assumed to satisfy the discrete inf-sup condition
for the stability of pressure, i.e., there is a constant $\beta$ independent of the mesh size h such that
$$\inf\limits_{q_h\in Q_h}\sup\limits_{\bv_h\in \bX_h}\frac{(q_h, \nabla\cdot\bv_h)}{\|\nabla\bv_h\|}\geq \beta>0.$$ 
The discretely divergence-free subspace of $\bX_h$ will be denoted by
\begin{eqnarray*}
\bV_h= \{ \bv_h \in \bX_h:\,\, (q_h, \nabla \cdot \bv_h)=0 \ \ \forall q_h \in Q_h \}\, .
\end{eqnarray*}
One of the most important properties of the space $\bV_h$ is to assure that the approximation properties of the spaces $\bX_h$ and $\bV_h$ versus continuous vector fields in $\bV$ are equivalent:
$$\inf\limits_{\bv_h\in \bV_h} \|\nabla(\bu -\bv_h)\|\leq C \inf\limits_{\bv_h\in \bX_h} \|\nabla(\bu -\bv_h)\|,\hspace{3mm}\forall\bu\in \bV, $$
where $C$ is dependent on $\beta$, but independent of the mesh size $h$.
For functions $v(t, \bx)$ defined on the entire time interval $(0, t^*]$, we define the following norms
\begin{align*}
\|v\|_{\infty, k}:=ess\sup\limits_{0< t< t^*}\|v(t, \cdot)\|_{k}, \hspace{3mm} \text{and} \hspace{3mm}\|v\|_{m, k}:=\bigg(\int\limits_{0}^{t^*}\|v(t, \cdot)\|_{k}^m\,  \mathrm{dt} \bigg)^{1/m},\hspace{2mm}1\leq m < \infty.
\end{align*}
We also introduce the notation $t^{n+1}:= (n+1)\,\Delta t$, where $\Delta t$ is a chosen time-step, and the following discrete time norms:
\begin{align*}
\||v|\|_{\infty, k}:=\max\limits_{0\leq n \leq N}\|v(t^n,\cdot)\|_{k}, \hspace{3mm} \text{and} \hspace{3mm}\||v|\|_{m, k}:=\bigg(\Delta t\, \sum\limits_{n=0}^{N-1} \|v(t, \cdot)\|_{k}^m  \bigg)^{1/m}.
\end{align*}
In our stability and convergence analysis, we often call Young's inequalities.
\begin{lemma}
Let $a,b$ be non-negative real numbers. Then for any $\varepsilon>0$
\begin{align*}
a\,b\leq \frac{\varepsilon}{p}a^{p} +\frac{\varepsilon^{-q/p}}{q}b^{q},
\end{align*}
where $\frac{1}{p}+ \frac{1}{q}=1$ with $p,q\in [1,\infty).$
\end{lemma}
In addition, we use the following identity in our analysis for ease in handling the time-derivative term. For any $a$, $b$, and $c$,
\begin{equation}
(3a-4b+c, a) = \frac{a^2+(2a-b)^2}{2} - \frac{b^2+(2b-c)^2}{2} + \frac{(a -2b +c)^2}{2}. \label{bdf2id}
\end{equation}
We will also use a discrete Gronwall Lemma in our convergence analysis. Note that this is not the usual discrete Gronwall, but a variation of it that does not require a time step restriction.
\begin{lemma}[Discrete Gronwall Lemma]\label{discreteGronwall}
Let $\Delta t$, H, and $a_{n},b_{n},c_{n},d_{n}$ (for integers $n
\geq 0$) be non-negative numbers such that 
\begin{equation}
a_{l}+\Delta t
\sum_{n=0}^{l} b_{n} \leq \Delta t \sum_{n=0}^{l-1} d_{n}a_{n} +
\Delta t\sum_{n=0}^{l}c_{n} + H \ \ for\ l\geq 0.
\end{equation}
Then for all $\Delta t>0$,
\begin{equation}
a_{l}+\Delta t\sum_{n=0}^{l}b_{n} \leq \exp\left( \Delta
t\sum_{n=0}^{l-1}d_{n}\right) \left( \Delta
t\sum_{n=0}^{l}c_{n} + H \right)\ \ for\ l \ge 0.
\end{equation}
\end{lemma}
\begin{proof}
The result can be found in \cite{HR90}.
\end{proof}
\begin{lemma}\label{consistency1}
Assume $\varphi$ sufficiently smooth. Then, the following holds 
\begin{align*}
\left\|\varphi_t^{n+1}-\frac{3\varphi^{n+1} - 4\varphi^{n} + \varphi^{n-1} }{2\Delta t}\right\|^2& \leq C \Delta t^3\int_{t^{n-1}}^{t^{n+1}}\|\varphi_{ttt}\|^2 dt,\\
\left\|\varphi^{n+1}- 2\varphi^{n} +\varphi^{n-1}\right\|^2& \leq C {\Delta t^3}\int_{t^{n-1}}^{t^{n+1}}\|\varphi_{tt}\|^2 dt.
\end{align*}
\end{lemma}
\section{Filtering}
In this section, we define the filtering and deconvolution operations used throughout this paper, and present some preliminary results. We assume that the filtering radius of the linear filter and non-linear filter from the model are same. This assumption is sensible since $\alpha$ is perceived as the smallest resolvable structures. Thus, the choice of $\alpha =O(h)$, where $h$ is the mesh width, is appropriate for both filters. We emphasize that our analysis can be extended in the case that these filtering radius are different, but still $O(h)$. 
\begin{remark}
Observe that differential filter of the velocity field $\bu$ requires a solution of a Helmholtz
equation for each component, separately. Therefore, all estimates for the discrete Helmholtz (dif-
ferential) is here presented for the scalar case. To avoid a confusion in the notation, we will use
the same representation of spaces $X := H^1_0(\Omega)$ and $V_h\subset X_h$ without bold notation.
\end{remark}
\begin{dfn}[Continuous Helmholtz Filter] Assume the filtering radius $\alpha>0$, and $\psi\in L^2(\Omega)$ be given. ${\widetilde{\psi}}$ is called the Helmholtz filter of $\psi$, if it fulfils the following equation
\begin{align}
\alpha^{2}(\nabla{\widetilde{\psi}},\nabla v) + ({\widetilde{\psi}}, v)= (\psi, v), \hspace{2mm}\forall v\in X. \label{confilter}
\end{align}
\end{dfn}
\noindent Similarly, we can define discrete Helmholtz filtering as follows.
\begin{dfn}[Discrete Helmholtz Filter] 
Let the filtering radius $\alpha>0$, and $\psi\in L^2(\Omega)$ be given. Then the discrete Helmholtz filter of $\psi$ is the solution ${\widetilde{\psi}}_h$ of the equation: $\forall v_h\in V_h,$
\begin{align}
\alpha^{2}(\nabla{\widetilde{\psi}}_h,\nabla v_h) + ({\widetilde{\psi}}_h, v_h) = (\psi, v_h).\label{disfilter}
\end{align}
\end{dfn}
\noindent Next, we define discrete and continuous van Cittert deconvolution. The filtering radius is selected as $\alpha = \mathcal{O}(h).$
\begin{dfn}
Let $F_h$ denote the filter given in \eqref{disfilter} such that $F_h{{\psi}}:= {\widetilde{\psi}}_h$. Then the continuous and discrete van Cittert deconvolution operators $D_{N}$ and $D_{N}^{h}$ are defined by
\begin{equation}
D_{N} := \sum_{n=0}^{N}
(I-F)^{n}\, , \qquad  D^{h}_{N} := \sum_{n=0}^{N}
(I-F_{h})^{n} .
\end{equation}
\label{deconvolution}
\end{dfn}
\noindent From \cite{DE06}, we know that $D_N$ acts as an approximate inverse to the filter $F$ in the following sense:
\begin{equation}
{{\psi}} - D_N{\widetilde{\psi}}= (-1)^{N+1} \alpha^{2N+2} \Delta^{N+1} F^{N+1}\phi. \label{deconerror}
\end{equation}
For the discrete deconvolution accuracy, we will utilize the following result from \cite{LMNR08}:
\begin{lemma}\label{discreterror}
For $\psi\in X \cap H^{2N+2}(\Omega)\cap H^{k+1}(\Omega)$,
\begin{equation}\label{DiscreteDeconvError}
\|\psi - D^{h}_{N}{\widetilde{\psi}}_h \| \leq C \left( (\alpha h^{k}+ h^{k+1}) \left( \sum_{n=0}^{N} | F^n {\widetilde{\psi}} |_{k+1} \right)
+ \alpha^{2N+2} \|  F^{N+1} \Delta^{N+1} \psi \| \right). 
\end{equation}
\end{lemma}
\begin{remark}
The dependence of the terms $| F^n {\widetilde{\psi}} |_{k+1}$ on the right side of \eqref{DiscreteDeconvError} on $\alpha $ is partially an open question. However, it is known from \cite{LMNR08, LR12} that in the periodic setting, they are independent of $\alpha$ and in the wall-bounded case, if $\Delta^j\psi=0$ on $\partial\Omega$ for $0\leq j\leq \lceil \frac{m}{2} \rceil -1$, then there exist $C_i$'s independent of $\alpha$ satisfying
\begin{align*}
\| F \psi \|_m &\le C_1\| \psi \|_m \ \ \ \ \ \ m=0,1,2, \\
\| F^2 \psi \|_m &\le C_2\| F \psi \|_m \ \ \ \ m=0,1,2,3,4 ,\\
\| F^3 \psi \|_m & \le C_3\| F^2 \psi \|_m \ \ \  m=0,1,2,3,4,5,6 
\end{align*}
and so on.  Thus for $k \ge 2$, if $\Delta^j\bm\phi=0$ on $\partial\Omega$ for $0\leq j\leq \lceil \frac{k+1}{2} \rceil -1$ and for any $N\ge 0$, the loss of a power of $\alpha$ cannot be ruled out, so we can only conclude
\[
\|\psi - D^{h}_{N}{\widetilde{\psi}}_h \| \leq C \left( h^k + \frac{h^{k+1}}{\alpha} + \alpha^{2N+2} \right) 
\]
with $C$ independent of $h$.
\end{remark}

\begin{dfn}
We define as an indicator function
\begin{equation}
a_{D_N}(\psi)(x):= |\psi(x)-D_N^h{\tilde{\psi}}_h(x)|.
\end{equation}
\end{dfn}
\begin{rem}
The function $a_{D_N}$ cannot be expected to satisfy $a_{D_N}(x)\leq 1$ for all $x$, however, this relation appears to be true for flows normalized to 1. If it is not, the following indicator could be used instead
\begin{equation}
\hat{a_{D_N}}(\psi)(x):= \frac{a_{D_N}(\psi)(x)}{max\{1,\|a_{D_N}(\psi)\|_{L^{\infty}(0,T;L^{\infty})}\}},
\end{equation}
and all theory would presented herein will hold true. Thus, without loss of generality, we assume $a_{D_N}(x)\leq 1$.
\end{rem}
\noindent Adaptive filtering with the deconvolution based indicator function is defined as follows.
\begin{dfn}
Let $u\in  L^2(\Omega)^d$, and an averaging radius $\alpha>0$ be given. Then the adaptively filtered velocity $ \overline{u}^h \in  V_h$ is the solution of the following equation: $\forall v_h\in V_h,$
\begin{align}
\alpha^2\left(a_{D_N}(u)\nabla{\overline{u}}^h,\nabla v_h\right) + \left(\overline{u}^h, v_h\right) & = \left(u, v_h\right).\label{af1}
\end{align} 
\end{dfn}
\begin{dfn}
Let $\psi\in L^2(\Omega)^d$, $u_h\in \bV_h$. Then ${\widehat{\psi}}^{u_h}\in V_h$ is defined to be the solution of the following equation:
\begin{align}
\alpha^2(a_{D_N}(u_h)\nabla{\widehat{\psi}}^{u_h}, \nabla  v_h) + (\widehat{\psi}^{u_h}, v_h)=(\psi, v_h), \hspace{2mm}\forall v_h\in V_h\label{definter}.
\end{align}
\end{dfn}
\noindent The next lemma gives and proves bounds on adaptively filtered variables, see \cite{BR13}).
\begin{lemma}\label{connect3}
Let $u \in V\cap H^{k+1}(\Omega)\cap {H}^3(\Omega)$, and $u_h\in V_h$, and $\alpha=\mathcal{O}(h)$. Assume $\widehat{u}^{u_h}$ satisfy \eqref{definter}. Then, we have the bounds
\begin{gather}
\alpha^2\|\sqrt{a_{D_N}(u_h)}\nabla\left( u - \widehat{u}^{u_h}\right)\|^2 + \| u - \widehat{u}^{u_h}\|^2 \leq C \alpha^2 \|u-u_h\|^2 + C\left(\alpha^2h^{2k} + \alpha^2\|a_{D_N}(u)\|^2 +  h^{2k+2}\right).
\end{gather}
\end{lemma}
\begin{lemma}\label{connect1}
For $u\in V$,  and $u_h\in V_h$
$$ \|\widehat{u}^{u_h} - \overline{u_h}^h\|\leq \|u-u_h\|.$$
\end{lemma}
\noindent The bounds on $\|u-\overline{u_h}^h\|$ plays a key role in our convergence analysis presented in Section 4 and Section 5. The following lemma gives a relation between the function and its adaptively filtered representation.
\begin{lemma}\label{connect2}
Let $u \in V\cap H^{k+1}(\Omega)\cap {H}^3(\Omega)$, and $u_h\in V_h$. Assume that $\alpha=\mathcal{O}(h)<1$. Then we have the bounds,
\begin{align}
\|u-\overline{u_h}^h\| & \leq \|u-\widehat{u}^{u_h}\| + \|\widehat{u}^{u_h} - \overline{u_h}^h\|\leq C (\|u-u_h\| + \alpha h^{k} + \alpha\|a_{D_N}(u)\| + h^{k+1} ).
\end{align}
\end{lemma}

\section{Numerical Scheme and Analysis}\label{bouconvsection}
Leray-Boussinesq model in $(0,t^*]\times \Omega$ with appropriate boundary and initial conditions are given by
\begin{eqnarray}
\begin{split}
\bu_t + \overline{\bu}\cdot\nabla \bu -Re^{-1}\Delta \bu +\nabla p - RiT\hat{\bk}&=&\bf,\\
\nabla\cdot \bu &=& 0 ,\\
T_t + \bu\cdot\nabla T-(RePr)^{-1}\Delta T &=& \gamma,\\
-\alpha^2 \nabla \cdot ( a(\bu) \nabla \overline{\bu} ) + \overline{\bu} + \nabla \lambda & = & \bu,\\
\nabla \cdot \overline{\bu} & = & 0.
\end{split} 
\label{leraybou}
\end{eqnarray}
\textit{It is clear to see that \eqref{leraybou} enforces the incompressibility condition on the filter
which is required outside of the periodic case to preserve the standard energy inequality and the well-posedness of the model, see \cite{RJH06}. Even if imposing the incompressibility constraint on $\bu$ leads
to more physically consistent method, this enforcement leads to an overdetermined system since the
indicator functions uses known velocity. Therefore, the nonlinear filter in \eqref{leraybou} uses a Lagrange multiplier to fix
that.
}
We now define the numerical scheme for approximating Leray-Boussinesq model.
\begin{alg}
\label{boualg}
Let body forces $\bf, \gamma$, initial conditions $\bu_h^{1}, \bu_h^0$ and $T_h^{1}, T_h^0 $ and filtering radius $\alpha\le \mathcal{O}(h)$ be given. Choose an end time $t^*>0$ and a time step $\Delta t$ such that $t^*=M \Delta t$. Then for $n=1,2,...,M$, find $(\bu_h^{n+1}, p_h^{n+1}, T_h^{n+1})\in (\bX_h,Q_h,Y_h)$ such that it holds, $\forall (\bv_h, q_h, s_h) \in (\bX_h, Q_h, Y_h),$ 
%
\begin{eqnarray}
\frac{1}{2\Delta t}(3T_h^{n+1} - 4T_h^n + T_h^{n-1}, s_h) + c\left(2\bu_h^n- \bu_h^{n-1}, T_h^{n+1}, s_h\right) && \nonumber \\
+(RePr)^{-1}(\nabla T_h^{n+1},\nabla s_h)  &=& (\gamma^{n+1},s_h), \label{bou3}\\
\frac{1}{2\Delta t}(3\bu_h^{n+1} - 4\bu_h^n + \bu_h^{n-1}, \bv_h) + b\left(\overline{\bm{\mathcal{U}}_h^n}^h, \bu_h^{n+1}, \bv_h\right)&& \nonumber \\
+ Re^{-1}(\nabla \bu_h^{n+1},\nabla \bv_h)
- (p_h^{n+1},\nabla \cdot \bv_h)-Ri((2T_h^{n}-T_h^{n-1})\hat{\bk}, \bv_h) &=&
(\bf^{n+1}, \bv_h), \quad \label{bou1}\\
(\nabla \cdot \bu_h^{n+1}, q_h)  & = &  0, \label{bou2}
\end{eqnarray}
where $\overline{\bm{\mathcal{U}}_h^n}^h:= \overline{2\bu_h^n - \bu_h^{n-1}}^h.$
\end{alg}
\begin{remark}
We emphasize here that the nonlinear filter step is linear at each time step since the indicator function uses known velocities. Therefore, linearization of the convective terms, with the second order extrapolation, makes Algorithm 4.1 linear at each time level. In this way, the momentum equation, heat equation and the filter step are all decoupled from one another. Therefore, one has to solve one adaptive filter step for the velocity and one for the usual Boussinesq system. This leads to no significant extra cost in computations if one compares this scheme with the usual Leray-$\alpha$ model. Notice that this extra cost is resulted from the calculation of $a(\bm u)$.
\end{remark}
\subsection{Stability}
This section is devoted to proving the stability of the numerical scheme. 
\begin{lem} \label{boustab}
Let $\bf\in L^{\infty}(0,T; \bH^{-1}(\Omega))$,  $\gamma\in L^{\infty}(0,T; H^{-1}(\Omega))$ and initial conditions $\bu_h^{1}, \bu_h^0 \in \bV_h$, $T_h^{1}, T_h^0 \in Y_h$. Then, solutions to Algorithm \ref{boualg} satisfy, $\forall \Delta t>0$,
\begin{multline}
\|\bu_h^{M}\|^2 + \|T_h^{M}\|^2 + \|2\bu_h^{M}- \bu_h^{M-1}\|^2 + \|2T_h^{M} - T_h^{M-1}\|^2
+ 2 Re^{-1}\Delta t\sum_{n=0}^{M-1}\|\nabla \bu_h^{n+1}\|^2\\
+ 2(Re Pr)^{-1}\Delta t\sum_{n=0}^{M-1} \|\nabla T_h^{n+1}\|^2  \\
\leq \|\bu_h^1\|^2 + \|T_h^1\|^2  + \|2\bu_h^{1}- \bu_h^{0}\|^2 + \|2T_h^{1} - T_h^{0}\|^2 + 4Re\Delta t \sum_{n=0}^{M-1}\|\bf^{n+1}\|^2_{-1} \\
+ 2 RePr\Delta t \sum_{n=0}^{M-1} \|\gamma^{n+1}\|_{-1}^2
 + 4\,  C_P^2\, Ri^2\, Re\,C_T\, t^* .
\end{multline}
\end{lem}
\begin{proof}
Set $s_h = T_h^{n+1}$ in \eqref{bou3}, which vanishes the non-linear term, then using \eqref{bdf2id} followed by the Cauchy-Schwarz and the Young's inequalities we get
\begin{multline*}
\frac{1}{4\Delta t}\big( \|T_h^{n+1}\|^2   - \|T_h^n\|^2  +  \|2T_h^{n+1} - T_h^n\|^2 - \|2T_h^{n} - T_h^{n-1}\|^2 +\|T_h^{n+1} - 2T_h^n + T_h^{n-1}\|^2\big)\\
  +(RePr)^{-1}\|\nabla T_h^{n+1}\|^2  =(\gamma^{n+1},T_h^{n+1})\\
\, \, \, \hspace{5cm} \leq \|\gamma^{n+1}\|_{-1}\|\nabla T_h^{n+1}\| \leq  \frac{RePr}{2}\|\gamma ^{n+1}\|_{-1}^2 + \frac{(RePr)^{-1}}{2}\|\nabla T_h^{n+1}\|^2 .
\end{multline*}
Reordering terms, multiplying by $4\Delta t$, dropping the fifth left hand side term, summing over time steps
yields
\begin{multline}
\|T_h^{M}\|^2 + \|2T_h^{M} - T_h^{M-1}\|^2 + 2\Delta t(RePr)^{-1} \sum_{n=0}^{M-1}\|\nabla T_h^{n+1}\|^2 \\
\leq \|T_h^1\|^2 + \|2T_h^{1} - T_h^{0}\|^2 + 2RePr\Delta t \sum_{n=0}^{M-1}\|\gamma^{n+1}\|^2_{-1}:=C_T(T_h^0, T_h^1, Re, Pr). \label{tempstab}
\end{multline}
Now set $\bv_h = \bu_h^{n+1}$ in \eqref{bou1}, $q_h = p_h^{n+1}$ in \eqref{bou2}. The non-linear and pressure terms vanish, and using the same identity as above gives
\begin{multline*}
\frac{1}{4\Delta t}\left( \|\bu_h^{n+1}\|^2 - \|\bu_h^n\|^2  + \|2\bu_h^{n+1}- \bu_h^n\|^2 - \|2\bu_h^{n}-\bu_h^{n-1}\|^2 + \|\bu_h^{n+1}-2\bu_h^n + \bu_h^{n-1}\|^2\right) \\
+ Re^{-1}\|\nabla \bu_h^{n+1}\|^2 = (\bf^{n+1}, \bu_h^{n+1}) + Ri((2T_h^{n}-T_h^{n-1})\hat{\bk}, \bu_h^{n+1}).
\end{multline*}
Using the Cauchy-Schwarz, Young, and Poincar{\'{e}}-Friedrichs' inequalities on the right hand side, and rearranging terms yields
\begin{multline*}
\frac{1}{4\Delta t}\left( \|\bu_h^{n+1}\|^2 - \|\bu_h^n\|^2  + \|2\bu_h^{n+1}- \bu_h^n\|^2 - \|2\bu_h^{n}-\bu_h^{n-1}\|^2 + \|\bu_h^{n+1}-2\bu_h^n + \bu_h^{n-1}\|^2\right) \\
+ \frac{Re^{-1}}{2}\|\nabla \bu_h^{n+1}\|^2
\leq  {Re}\|\bf^{n+1}\|_{-1}^2  + C_P^2\, Ri^2 Re \|2T_h^{n}-T_h^{n-1}\|^2.
\end{multline*}
Using the bound \eqref{tempstab} on the last term on the right hand side, multiplying by $4\Delta t$ and summing over time steps produces
\begin{multline}
\|\bu_h^{M}\|^2 + \|2\bu_h^{M}- \bu_h^{M-1}\|^2 + 2Re^{-1}\Delta t\sum_{n=0}^{M-1}\|\nabla \bu_h^{n+1}\|^2\\
\leq \|\bu_h^1\|^2 + \|2\bu_h^{1}- \bu_h^{0}\|^2 + 4 Re\Delta t \sum_{n=0}^{M-1}\|\bf^{n+1}\|^2_{-1} + 4\,C_P^2\, Ri^2\, Re C_T\, t^*\label{velstab1}. 
\end{multline}
Adding \eqref{velstab1} to \eqref{tempstab} completes the proof.
\end{proof}

\begin{rem}
This result immediately implies that solutions to Algorithm \ref{boualg} exist uniquely.
\end{rem}

\subsection{Convergence}
For simplicity in stating the following theorem, we state here the regularity assumptions of the solution $(u(x,t),p(x,t),T(x,t))$ of the true Boussinesq solutions:
\begin{eqnarray*}
&&\bu\in L^{\infty}(0,t^*; \bH^{k+1}\cap \bV\cap \bH^3(\Omega)),\hspace{1mm}  \\
&& \bu_{tt} \in L^2(0,t^*; \bH^1(\Omega)),\hspace{1mm}, \bu_{ttt}\in L^2(0,t^*; \bL^2(\Omega)),\hspace{1mm}\\
&& T\in L^{\infty}(0,t^*; H^{k+1}\cap V\cap H^3(\Omega)),\\
&&T_{tt},  T_{ttt} \in L^2(0,t^*; L^2(\Omega)), \\
&& p\in L^{\infty}(0,t^*;H^k(\Omega)). 
\end{eqnarray*}
\begin{theorem}\label{bouscon}
Let $(\bu_{h}^{n}, \, p_{h}^{n},\, T_h^n)$, $n=0,\, 1,\, \ldots M$, be the solution of Algorithm \eqref{boualg}, and $(\bu(t), p(t), T(t))$ be a solution of the Boussinesq equations satisfying no-slip boundary and the regularity conditions. Then using $(\bP_{k}, \, P_{k-1}, P_k)$ or $(\bP_k, \, P_{k-1}^{disc}, P_k)$ finite elements, the errors satisfy the bound, for any $\Delta t>0$
\begin{multline}
\|\bu(t^*)-u_h^{M}\|^2 +\|T(t^*)-T_h^{M}\|^2 + Re^{-1}\Delta t\sum_{n=0}^{M-1} \|\nabla( \bu^{n+1} - \bu_h^{n+1})\|^2\\
 + (RePr)^{-1} \Delta t\sum_{n=0}^{M-1}\|\nabla(T^{n+1}-T_h^{n+1})\|^2\leq C\bigg( \Delta t^4+ \alpha^2h^{2k} + \alpha^2\|a_{D_N}(\bu)\|^2 + h^{2k}\bigg),
\end{multline}
where $C$ is a constant independent of $\alpha$, $h$ and $\Delta t$.
\end{theorem}
\begin{rem}\label{BouRate}
If we assume the periodic setting, or $\Delta^j \bu = 0$ on $\partial\Omega$ for $0\leq j\leq \lceil\frac{m}{2}\rceil-1$, then the result becomes
\begin{multline}
\|\bu(t^*)-\bu_h^{M}\|^2 +\|T(t^*)-T_h^{M}\|^2 + Re^{-1}\Delta t\sum_{n=0}^{M-1} \|\nabla( \bu^{n+1}-\bu_h^{n+1})\|^2\\
 + (RePr)^{-1} \Delta t\sum_{n=0}^{M-1}\|\nabla(T^{n+1}-T_h^{n+1})\|^2\leq C\bigg( \Delta t^4 + \alpha^2h^{2k} + \alpha^{4N+6} + h^{2k}\bigg).
\end{multline}
\end{rem}
\begin{proof} The proof is divided into three steps since it is very long and technical. In the first step, the error equations are obtained by splitting the velocity and temperature errors into approximation errors and finite element remainders. In the second step, all right hand side terms of the error equations are bounded below. In the third step, the Gronwall Lemma and the triangle inequality are applied to the error terms.\ \\ 

\noindent \textbf{\textit{Step 1. [The derivation of error equations]}} \ \\

\noindent The true solutions of the Boussinesq system at time $t=t^{n+1}$ satisfies the following variational formulations, $\forall (\bv_h, q_h, s_h)\in (\bV_h, Q_h, Y_h)$
\begin{eqnarray}
\begin{split}
\left(\frac{3\bu^{n+1} - 4\bu^n + \bu^{n-1}}{2\Delta t}, \bv_h\right) + Re^{-1}(\nabla \bu^{n+1},\nabla \bv_h) + b\left(\widehat{\bm{\mathcal{U}}^n}^{\bm{\mathcal{U}}_h^n}, \bu^{n+1}, \bv_h\right) \\ 
-Ri((2T^n - T^{n-1})\bm{\hat{k}}, \bv_h)-(p^{n+1} ,\nabla\cdot \bv_h) = (\bf^{n+1}, \bv_h) + G(\bu, T, \bv_h), \label{BOU1}
\end{split}\\
\begin{split}
\left(\frac{3T^{n+1}-4T^n + T^{n-1}}{2\Delta t}, s_h\right)+ (RePr)^{-1}(\nabla T^{n+1},\nabla s_h) + c\left(2\bu^n - \bu^{n-1}, T^{n+1}, s_h\right)\\
 = (\gamma^{n+1}, s_h) + F(\bu, T, s_h),\label{BOU2}
\end{split}
\end{eqnarray}
where $\bu^{n}:=\bu(t^n), \,\,p^n:=p(t^n), \,\, T^{n}:=T(t^n), \,\, \bf^n:=\bf(t^n), \,\,\gamma^n:=\gamma(t^n), $
$n=0,1,...,M$, and $\bm{\mathcal{U}}^n:=2\bu^n-\bu^{n-1}$, $\bm{\mathcal{U}}^n_h:=2\bu^n_h-\bu^{n-1}_h$ and 
\begin{eqnarray}
G(\bu, T, \bv_h) &:=& \left(\frac{3\bu^{n+1} - 4\bu^n + \bu^{n-1}}{2\Delta t} - \bu_t^{n+1}, \bv_h\right)  \nonumber \\
\, \, \, & + & b\left(\widehat{\bm{\mathcal{U}}^n}^{\bm{\mathcal{U}}_h^n}, \bu^{n+1}, \bv_h\right) - b(\bu^{n+1}, \bu^{n+1}, \bv_h) \nonumber \\
&+& Ri((T^{n+1}-2T^n + T^{n-1})\hat{\bk}, \bv_h), \nonumber\\
F(\bu, T, s_h) &:=& \left(\frac{3T^{n+1}-4T^n + T^{n-1}}{2\Delta t}-T_t^{n+1}, s_h\right)  \nonumber \\
&-& c\left(\bu^{n+1} - 2\bu^n + \bu^{n-1}, T^{n+1}, s_h\right)\nonumber.
\end{eqnarray}
Take $\bv_h\in \bV_h$ in (\ref{bou1}) which vanishes the pressure term, and then subtract (\ref{BOU1}) from (\ref{bou1}), (\ref{BOU2}) from (\ref{bou3}), rewrite the non-linear terms to get
\begin{align}
\begin{split}
\frac{1}{2 \Delta t}(3\be_\bu^{n+1} - 4 \be_\bu^n + \be_\bu^{n-1}, \bv_h)
+ b\left(\overline{\bm{\mathcal{U}}_h^n}^h, \be_\bu^{n+1}, \bv_h\right) + b\left(\overline{\bm{\mathcal{U}}_h^n}^h-\widehat{\bm{\mathcal{U}}^n}^{\bm{\mathcal{U}}_h^n}, \bu^{n+1}, \bv_h\right) \\
+ Re^{-1}(\nabla \be_\bu^{n+1},\nabla \bv_h)-Ri((2e_T^{n}-e_T^{n-1})\hat{\bk}, \bv_h) + (p^{n+1},\nabla\cdot \bv_h) = -G(\bu,T, \bv_h),
\end{split}
\label{errvel}\\
\begin{split}
\frac{1}{2 \Delta t}(3e_T^{n+1} -4 e_T^n + e_T^{n-1}, s_h)
+ c\left(2\be_\bu^n - \be_\bu^{n-1}, T^{n+1} , s_h\right) + c\left(2 \bu_h^n - \bu_h^{n-1}, e_T^{n+1}, s_h\right)\\
(RePr)^{-1}(\nabla e_T^{n+1}, \nabla s_h) = - F(\bu, T, s_h),
\end{split}
\label{errtemp}
\end{align}
where $\be_{\bu}^n:={\bu}_h^n-{\bu}^n$ and $e_T^n:=T_h^n-T^n$, $n=0,1,..., M$, are the velocity and temperature errors. Split these errors as follows
$$\be_\bu^{n}=(\bu_{h}^{n}-P_{\bV_{h}}^{L^2}(\bu^{n})) -(\bu^{n}-P_{\bV_{h}}^{L^2}(\bu^{n}))=:\bphi_{\bu,h}^{n}-\bta_\bu^{n},$$
$$e_T^{n}=(T_{h}^{n}-P_{Y_{h}}^{L^2}(T^{n})) - (T^{n}-P_{Y_{h}}^{L^2}(T^{n}))=:\phi_{T,h}^{n}- \eta_T^{n}.$$
Set $\bv_h=\bphi_{\bu,h}^{n+1}$ in \eqref{errvel}, $s_h=\phi_{T,h}^{n+1}$ in \eqref{errtemp}, and use \eqref{bdf2id} to get, $\forall q_h\in Q_h$
\begin{multline}
\frac{1}{4\Delta t}\left[\|\bphi_{\bu,h}^{n+1}\|^2  - \|\bphi_{\bu,h}^{n}\|^2 \right] + \frac{1}{4\Delta t}\left[\|2\bphi_{\bu,h}^{n+1}-\bphi_{\bu,h}^{n}\|^2 - \|2\bphi_{\bu,h}^{n}-\bphi_{\bu,h}^{n-1}\|^2 \right] + \frac{1}{4\Delta t}\|\bphi_{\bu,h}^{n+1} - 2\bphi_{\bu,h}^{n} + \bphi_{\bu,h}^{n-1}\|^2\\ 
+ Re^{-1} \|\nabla \bphi_{\bu,h}^{n+1}\|^2  \\
= Re^{-1} (\nabla \bta_\bu^{n+1}, \nabla \bphi_{\bu,h}^{n+1}) - (p^{n+1}-q_h, \nabla\cdot\bphi_{\bu,h}^{n+1}) + b\left(\widehat{\bm{\mathcal{U}}^n}^{\bm{\mathcal{U}}_h^n}-\overline{\bm{\mathcal{U}}_h^n}^h, \bu^{n+1},  \bphi_{\bu,h}^{n+1}\right) \\
 + b\left(\overline{\bm{\mathcal{U}}_h^n}^h, \bta_\bu^{n+1}, \bphi_{\bu,h}^{n+1}\right) + Ri(((2\eta_T^{n}-\eta_T^{n-1})- (2\phi_{T,h}^{n}-\phi_{T,h}^{n-1}))\hat{\bk},\bphi_{\bu,h}^{n+1}) 
 +  G(\bu, T,  \bphi_{\bu,h}^{n+1}),\label{bouerru1}
\end{multline}
and 
\begin{multline}
\frac{1}{4\Delta t}\left[\|\phi_{T,h}^{n+1}\|^2 - \|\phi_{T,h}^{n}\|^2 \right] + \frac{1}{4\Delta t}\left[\|2\phi_{T,h}^{n+1}-\phi_{T,h}^{n}\|^2 - \|2\phi_{T,h}^{n}-\phi_{T,h}^{n-1}\|^2\right] + \frac{1}{4\Delta t}\|\phi_{T,h}^{n+1}-2\phi_{T,h}^{n}-\phi_{T,h}^{n-1}\|^2\\
 +{(RePr)^{-1}}\|\nabla\phi_{T,h}^{n+1}\|^2 \\
=(RePr)^{-1}(\nabla \eta_T^{n+1}, \nabla \phi_{T,h}^{n+1}) + c(\bm{\mathcal{U}}_h^n, \eta_T^{n+1},\phi_{T,h}^{n+1})
+c((2\bta_\bu^n-\bta_\bu^{n-1}) - (2\bphi_{\bu,h}^n - \bphi_{\bu,h}^{n-1}),\, T^{n+1}, \phi_{T,h}^{n+1})\\
+F(\bu, T,\phi_{T,h}^{n+1}). \label{bouerrT}
\end{multline}

\noindent \textbf{\textit{Step 2. [Estimations of right hand side terms]}} \ \\

\noindent Apply the Cauchy-Schwarz, the Young's and Poincar{\'{e}}-Friedrichs' Inequalities on the first two, and the fifth right hand side terms of \eqref{bouerru1} to get
\begin{multline}
\frac{1}{4\Delta t}\left[\|\bphi_{\bu,h}^{n+1}\|^2  - \|\bphi_{\bu,h}^{n}\|^2 \right] + \frac{1}{4\Delta t}\left[\|2\bphi_{\bu,h}^{n+1}-\bphi_{\bu,h}^{n}\|^2 - \|2\bphi_{\bu,h}^{n}-\bphi_{\bu,h}^{n-1}\|^2\right] + \frac{1}{4\Delta t}\|\bphi_{\bu,h}^{n+1} - 2\bphi_{\bu,h}^{n} + \bphi_{\bu,h}^{n-1}\|^2\\ + \frac{Re^{-1}}{2} \|\nabla \bphi_{\bu,h}^{n+1}\|^2 \\
\leq 2 Re^{-1} \|\nabla \bta_\bu^{n+1}\|^2 + 2 Re\inf\limits_{q_h\in Q_h}\|p^{n+1}-q_h\|^2
+ b(\widehat{\bm{\mathcal{U}}^n}^{\bm{\mathcal{U}}_h^n}-\overline{\bm{\mathcal{U}}_h^n}^h, \bu^{n+1}, \bphi_{\bu,h}^{n+1}) \\ 
+ b(\overline{\bm{\mathcal{U}}_h^n}^h, \bta_\bu^{n+1}, \bphi_{\bu,h}^{n+1}) + 2 Re C_P^2 Ri^2(\|2\eta_T^n-\eta_T^{n-1}\|^2+\|2\phi_{T,h}^{n}-\phi_{T,h}^{n-1}\|^2)+ G(\bu, T, \bphi_{\bu,h}^{n+1}).\label{bouerru2}
\end{multline}
Similar steps applied to the temperature equation yield 
\begin{multline}
\frac{1}{4\Delta t}\left[\|\phi_{T,h}^{n+1}\|^2 - \|\phi_{T,h}^{n}\|^2\right] + \frac{1}{4\Delta t}\left[\|2\phi_{T,h}^{n+1}-\phi_{T,h}^{n}\|^2 - \|2\phi_{T,h}^{n}-\phi_{T,h}^{n-1}\|^2\right] + \frac{1}{4\Delta t}\|\phi_{T,h}^{n+1}-2\phi_{T,h}^{n}-\phi_{T,h}^{n-1}\|^2\\
 +\frac{(RePr)^{-1}}{2}\|\nabla\phi_{T,h}^{n+1}\|^2 \\
\leq \frac{(RePr)^{-1}}{2}\|\nabla\eta_T^{n+1}\|^2+ c(\bm{\mathcal{U}}_h^n,\eta_T^{n+1}, \phi_{T,h}^{n+1})
+c((2\bta_\bu^n-\bta_\bu^{n-1}) - (2\bphi_{\bu,h}^n-\bphi_{\bu, h}^{n-1}), T^{n+1} ,\phi_{T, h}^{n+1}) \\
+F(\bu, T, \phi_{T,h}^{n+1}). \label{bouerrT}
\end{multline}
To bound the first non-linear term in \eqref{bouerru2}, we first expand $b(\cdot, \cdot, \cdot)$, and use Lemma~\ref{trilinearbound} followed by the Poincar{\'{e}}-Friedrichs' Inequality and Lemma~\ref{Agmon}. Then, we apply the Young's Inequality together with Lemma 3.3, which yields
\begin{alignat}{2}
&b\left(\widehat{\bm{\mathcal{U}}^n}^{\bm{\mathcal{U}}_h^n}-\overline{\bm{\mathcal{U}}_h^n}^h,  \bu^{n+1}, \bphi_{\bu,h}^{n+1}\right)\nonumber\\
&=\frac{1}{2}\left[\left((\widehat{\bm{\mathcal{U}}^n}^{\bm{\mathcal{U}}_h^n}-\overline{\bm{\mathcal{U}}_h^n}^h)\cdot\nabla  \bu^{n+1}, \bphi_{\bu,h}^{n+1}\right)- \left((\widehat{\bm{\mathcal{U}}^n}^{\bm{\mathcal{U}}_h^n}-\overline{\bm{\mathcal{U}}_h^n}^h)\cdot\nabla\bphi_{\bu,h}^{n+1}, \bu^{n+1}\right) \right]\nonumber\\
&\leq \frac{1}{2} \|\widehat{\bm{\mathcal{U}}^n}^{\bm{\mathcal{U}}_h^n}-\overline{\bm{\mathcal{U}}_h^n}^h\|\left(\|\nabla\bu^{n+1}\|_{L^\infty}C_P \|\nabla \bphi_{\bu,h}^{n+1}\| + \|\bu^{n+1}\|_{L^\infty}\|\nabla \bphi_{\bu,h}^{n+1}\|\right)\nonumber\\
&\leq C \|\widehat{\bm{\mathcal{U}}^n}^{\bm{\mathcal{U}}_h^n}-\overline{\bm{\mathcal{U}}_h^n}^h\|\left(\|\bu^{n+1}\|_{H^3}\|\nabla \bphi_{\bu,h}^{n+1}\| + \|\bu^{n+1}\|_{H^2}\|\nabla \bphi_{\bu,h}^{n+1}\|\right)\nonumber\\
&\leq C Re\left( \|\bta_{\bu}^{n}\|^2  + \|\bphi_{\bu,h}^{n}\|^2 + \|\bta_{\bu}^{n-1}\|^2 + \|\phi_{\bu,h}^{n-1}\|^2  \right)\| \bu^{n+1}\|^2_{H^3} +\frac{Re^{-1}}{20}\|\nabla\bphi_{\bu,h}^{n+1}\|^2 \label{vel1}.
\end{alignat}
We estimate the second non-linear term in a similar manner: apply Lemma 2.1 together with Lemma 2.2 and the inverse inequality, then approximation property of the true velocity solution, the Young's Inequality together with Lemma 3.4, which produces
\begin{align}
& b\left(\overline{\bm{\mathcal{U}}_h^n}^h,\bta_\bu^{n+1}, \bphi_{\bu,h}^{n+1}\right) =  
b\left(\overline{\bm{\mathcal{U}}_h^n}^h - \bm{\mathcal{U}}^n, \bta_\bu^{n+1}, \bphi_{\bu,h}^{n+1}\right)
+ b\left(\bm{\mathcal{U}}^n, \bta_\bu^{n+1}, \bphi_{\bu,h}^{n+1}\right) 
\nonumber\\
&\leq C\|\overline{\bm{\mathcal{U}}_h^n}^h - \bm{\mathcal{U}}^n\| \left(\|\nabla\bta_\bu^{n+1}\| \|\bphi_{\bu,h}^{n+1}\|_{L^{\infty}} + \|\bta_\bu^{n+1}\| \|\nabla\bphi_{\bu,h}^{n+1}\|_{L^{\infty}}  \right) + C\|\nabla \bm{\mathcal{U}}^n\| \|\nabla \bta_\bu^{n+1}\| \|\nabla\bphi_{\bu,h}^{n+1}\| \nonumber\\
&\leq C\|\overline{\bm{\mathcal{U}}_h^n}^h - \bm{\mathcal{U}}^n\| \left(\|\nabla\bta_\bu^{n+1}\| Ch^{-1/2} \|\nabla\bphi_{\bu,h}^{n+1}\| + \|\bta_\bu^{n+1}\|Ch^{-3/2} \|\nabla\bphi_{\bu,h}^{n+1}\|\right)
+ C\|\nabla \bm{\mathcal{U}}^n\| \|\nabla \bta_\bu^{n+1}\| \|\nabla\bphi_{\bu,h}^{n+1}\|\nonumber\\
&\leq C\|\overline{\bm{\mathcal{U}}_h^n}^h - \bm{\mathcal{U}}^n\| h^{k-1/2} \|\bu^{n+1}\|_{k+1}  \|\nabla\bphi_{\bu,h}^{n+1}\| + CRe\|\nabla \bm{\mathcal{U}}^n\|^2 \|\nabla \bta_\bu^{n+1}\|^2 + \frac{3Re^{-1}}{40}\|\nabla\bphi_{\bu,h}^{n+1}\|^2 \nonumber\\
&\leq C h^{2k-1}Re( \|{\bm{\mathcal{U}}_h^n} - \bm{\mathcal{U}}^n\|^2 + \alpha^2 h^{2k} + \alpha^2\|a_{D_N}(\bu)\|^2 ) + CRe\|\nabla \bm{\mathcal{U}}^n\|^2 \|\nabla \bta_\bu^{n+1}\|^2 + \frac{3Re^{-1}}{40}\|\nabla\bphi_{\bu,h}^{n+1}\|^2  \nonumber\\
&\leq C Re(\|\bta_\bu^{n-1}\|^2 + \|\bta_\bu^n\|^2+\|\bphi_{\bu,h}^{n-1}\|^2 + \|\bphi_{\bu,h}^n\|^2 + \alpha^2h^{2k}+\alpha^2\|a_{D_N}(\bu)\|^2)\nonumber\\
& \hspace{5cm}  + CRe\|\nabla \bm{\mathcal{U}}^n\|^2 \|\nabla \bta_\bu^{n+1}\|^2 +\frac{3Re^{-1}}{40}\|\nabla\bphi_{\bu,h}^{n+1}\|^2   \label{vel2}.
\end{align}
For the first and the fourth terms of the consistency error $G(\bu, T, \bphi_{\bu,h}^{n+1})$, we use Cauchy-Schwarz followed by the Poincar{\'{e}}-Friedrichs' Inequality, Lemma 2.5 and Young's inequality to get 
\begin{align}
\begin{split}
\left |\left( \frac{3\bu^{n+1}-4\bu^{n}+\bu^{n-1}}{2\Delta t}- \bu_{t}^{n+1}, \, \bm \phi_{\bu, h}^{n+1}\right)\right| & \leq \|\frac{3\bu^{n+1}-4\bu^{n}+\bu^{n-1}}{2\Delta t}-\bu_{t}^{n+1} \|C_P\|\nabla \bm \phi_{\bu, h}^{n+1}\|\\
& \leq  C Re (\Delta t)^3 \int\limits_{t^{n-1}}^{t^{n+1}}\|\bu_{ttt}\|^2 dt + \frac{Re^{-1}}{40}\|\nabla \bm \phi_{\bu, h}^{n+1}\|^2, 
\end{split}\label{vel3}
\\
\begin{split}
\left |Ri((T^{n+1} - 2T^n + T^{n-1})\bm{\hat{k}}, \bm \phi_{\bu, h}^{n+1})\right |& \leq Ri \|T^{n+1}-2T^n + T^{n-1}\|C_P\| \|\nabla \bm \phi_{\bu, h}^{n+1}\|\\
& \leq  C Re Ri^2 (\Delta t)^3 \int\limits_{t^{n-1}}^{t^{n+1}}\|T_{tt}\|^2 dt + \frac{Re^{-1}}{40}\|\nabla \bm \phi_{\bu, h}^{n+1}\|^2.
\end{split}
\label{vel4}
\end{align}
Assuming $\alpha \leq 1$ and expanding $\bm{\mathcal{U}}^n - \bm{\mathcal{U}}^n_h$, the remaining terms are estimated below as follows:
\begin{align}
 b(& \widehat{\bm{\mathcal{U}}^n}^{\bm{\mathcal{U}}_h^n},  \bu^{n+1}, \bm \phi_{\bu, h}^{n+1}) - b(\bu^{n+1}, \bu^{n+1}, \bm \phi_{\bu, h}^{n+1})\nonumber\\
& =  b\left(\widehat{\bm{\mathcal{U}}^n}^{\bm{\mathcal{U}}_h^n} - \bm{\mathcal{U}}^n, \bu^{n+1}, \bm \phi_{\bu, h}^{n+1}\right) - b\left(\bu^{n+1} - 2\bu^{n} + \bu^{n-1}, \bu^{n+1}, \bm \phi_{\bu, h}^{n+1}\right)\nonumber\\
& \leq C\|\widehat{\bm{\mathcal{U}}^n}^{\bm{\mathcal{U}}_h^n} - \bm{\mathcal{U}}^n\| \|\bu^{n+1}\|_{H^3}\|\nabla \bm \phi_{\bu, h}^{n+1}\| + C\|\nabla(\bu^{n+1} - 2\bu^{n} + \bu^{n-1})\|\|\nabla\bu^{n+1}\| \|\nabla \bm \phi_{\bu, h}^{n+1}\|\nonumber\\
& \leq C Re\left[\|\bm{\mathcal{U}}^n - \bm{\mathcal{U}}^n_h\|^2 +\alpha^2 h^{2k} + \alpha^2\|a_{D_N}(\bu)\|^2 + h^{2k+2}\right]\|\bu^{n+1}\|_{H^3}^2 \nonumber\\
& \,\,\,\, \, + C Re (\Delta t)^3 \|\nabla\bu^{n+1}\|^2  \int\limits_{t^{n-1}}^{t^{n+1}}\|\nabla\bu_{tt}\|^2 dt + \frac{3Re^{-1}}{40}\|\nabla \bm \phi_{\bu, h}^{n+1}\|^2 \nonumber\\
& \leq C Re\left[ \|\bm \phi_{\bu, h}^{n}\|^2 + \|\bm \phi_{\bu, h}^{n-1}\|^2 + \|\bta_{\bu}^{n}\|^2 +  \|\bta_{\bu}^{n-1}\|^2+ \alpha^2 h^{2k} + \alpha^2\|a_{D_N}(\bu)\|^2 + h^{2k+2}\right]\|\bu^{n+1}\|_{H^3}^2 \nonumber\\
& \,\,\,\, \, + C Re (\Delta t)^3 \|\nabla\bu^{n+1}\|^2  \int\limits_{t^{n-1}}^{t^{n+1}}\|\nabla\bu_{tt}\|^2 dt + \frac{3Re^{-1}}{40}\|\nabla \bm \phi_{\bu, h}^{n+1}\|^2.\label{vel5}
\end{align}
In a similar manner, we have the bounds for the non-linear terms in \eqref{bouerrT}
\begin{align}
c\left(\bm{\mathcal{U}}_h^n,\eta_T^{n+1},\phi_{T,h}^{n+1}\right)&\leq \|\nabla \bm{\mathcal{U}}_h^n\|\|\nabla\eta_T^{n+1}\|\|\nabla\phi_{T,h}^{n+1}\| \nonumber\\
& \leq \frac{(RePr)^{-1}}{24}\|\nabla\phi_{T,h}^{n+1}\|^2+CRePr( \|\nabla \bm{\mathcal{U}}_h^n\|^2\|\nabla\eta_T^{n+1}\|^2)\label{temp1},
\end{align}
and
\begin{multline}
c\left(2\bta_\bu^n-\bta_\bu^{n-1},T^{n+1},\phi_{T,h}^{n+1}\right)\leq C\|\nabla(2\bta_\bu^n-\bta_\bu^{n-1})\|\|\nabla T^{n+1}\|\|\nabla\phi_{T,h}^{n+1}\|\\
\leq \frac{(RePr)^{-1}}{24}\|\nabla\phi_{T,h}^{n+1}\|^2 +CRePr\|\nabla T^{n+1}\|^2\|\nabla(2\bta_\bu^n-\bta_\bu^{n-1})\|^2\label{temp2},
\end{multline}
and
\begin{alignat}{2}
 c\left(2\bphi_{\bu,h}^n-\bphi_{\bu,h}^{n-1},T^{n+1},\phi_{T,h}^{n+1}\right)&=\frac{1}{2}\left[((2\bphi_{\bu,h}^n-\bphi_{\bu,h}^{n-1})\cdot \nabla T^{n+1}, \phi_{T,h}^{n+1}) 
-((2\bphi_{\bu,h}^n-\bphi_{\bu,h}^{n-1})\cdot \nabla\phi_{T,h}^{n+1},T^{n+1})\right] \nonumber\\
& \leq C\|2\bphi_{\bu,h}^n-\bphi_{\bu,h}^{n-1}\|( \|\nabla T^{n+1}\|_{\infty} C_P\|\nabla \phi_{T,h}^{n+1}\| + \| T^{n+1}\|_{L^\infty}\|\|\nabla \phi_{T,h}^{n+1}\|)\nonumber\\
& \leq \frac{(RePr)^{-1}}{12}\|\nabla \phi_{T,h}^{n+1}\|^2 + CRePr\|2\bphi_{\bu,h}^n-\bphi_{\bu,h}^{n-1}\|^2\|T^{n+1}\|_{H^3}^2 \label{temp3}.
\end{alignat}
The terms in consistency error $F(\bu, T, \phi_{T,h}^{n+1})$ are bounded below as follows:
\begin{align}
|F(\bu, T, \phi_{T,h}^{n+1})|\leq \frac{(RePr)^{-1}}{12}\|\nabla \phi_{T, h}^{n+1}\|^2 + C (Re Pr)(\Delta t)^3\left( \int\limits_{t^{n-1}}^{t^{n+1}}\|T_{ttt}\|^2 dt + \|\nabla T^{n+1}\|^2  \int\limits_{t^{n-1}}^{t^{n+1}}\|\nabla\bu_{tt}\|^2 dt \right)\label{temp4}.
\end{align} 

\noindent \textbf{\textit{Step 3. [The application of the Gronwall Lemma and the triangle inequality]}}\ \\
 
\noindent Combining all bounds \eqref{vel1}-\eqref{vel5} with \eqref{bouerru2}, multiplying by $4\Delta t$, summing over time step, and reducing gives
\begin{multline}
\|\bphi_{\bu,h}^{M}\|^2 + \|2\bphi_{\bu,h}^{M}-\bphi_{\bu,h}^{M-1}\|^2
+ {Re^{-1}}\Delta t \sum_{n=0}^{M-1}\|\nabla \bphi_{\bu,h}^{n+1}\|^2 \\
\leq  \|\bphi_{\bu,h}^{0}\|^2 + \|2\bphi_{\bu,h}^{1}-\bphi_{\bu,h}^{0}\|^2
+ C \Delta t \sum_{n=0}^{M-1}({Re^{-1}} + {Re})\|\nabla \bta_\bu^{n+1}\|^2 + 2 Re\inf\limits_{q_h\in Q_h}\|p^{n+1}-q_h\|^2 \\
+ C Re\Delta t \sum_{n=0}^{M-1} (2 + \|\bu^{n+1}\|^2_{H^3} )\left(\|\bta_{\bu}^{n}\|^2 +  \|\bta_{\bu}^{n-1}\|^2 + \|\bm \phi_{\bu, h}^{n}\|^2 + \|\bm \phi_{\bu, h}^{n-1}\|^2\right)\\
+ C Re\, t^*(1+ \||\bu\||_{\infty, 3}^2) (\alpha^2 h^{2k} + \alpha^2 \|a_{D_N}(\bu)
\|^2 +h^{2k+2})\\
+ C Re C_P^2Ri^2\Delta t \sum_{n=0}^{M-1}(\|\eta_T^n \|^2 + \|\eta_T^{n-1}\|^2 +\|\phi_{T,h}^{n}\|^2 +\phi_{T,h}^{n-1}\|^2)\\
+ CRe(\Delta t)^4\left(\int_{0}^{t^*}\|\bu_{ttt}\|^2 dt + \int_{0}^{t^*}\|T_{tt}\|^2 dt  +  \||\nabla\bu\||_{\infty, 0}^2 \int_{0}^{t^*}\|\nabla\bu_{tt}\|^2dt \right),
\end{multline}
and from which applying regularity assumptions produces
\begin{multline}
\|\bphi_{\bu,h}^{M}\|^2 + Re^{-1} \Delta t\sum_{n=0}^{M-1} \|\nabla \bphi_{\bu,h}^{n+1}\|^2 \\
\leq C\, h^{2k}(Re^{-1} + Re \||\bu\||_{\infty, 1}^2 )\||\bu\| |^2_{k+1}
+C Re h^{2k}\||p\||_{2,k}^2 + C Re h^{2k+2}\|||\bu\||_{2, k+1}^2 \\
+ C\, h^{2k+2} Re (2 + \||\bu\||_{\infty, 3}^2)\||\bu\||_{2, k+1}^2
+C Re \Delta t\sum_{n=0}^{M-1}(2 + \|\bu^{n+1}\|^2_{H^3})\|\bphi_{\bu,h}^n\|^2
+C Re C_P^2 Ri^2 h^{2k+2}\||T\||_{2, k+1}^2\\
+C Re C_P^2 Ri^2 \Delta t\sum_{n=0}^{M-1}\|\phi_{T,h}^n\|^2
+ C t^* Re  (\alpha^2h^{2k} + \alpha^2\|a_{D_N}(\bu)\|^2 + h^{2k+2} )\\
+ C\Delta t^4 (\|\bu_{ttt}\|_{2,0}^2 + \|T_{tt}\|_{2,0}^2 +  \||\nabla\bu\||_{\infty, 0}^2 \|\nabla\bu_{tt}\|_{2,0}^2). \label{sumeq1}
\end{multline}
Similarly, plugging estimates \eqref{temp1}-\eqref{temp4} into \eqref{bouerrT}, summing over time steps, multiplying by $4\Delta t$, using regularity assumptions and rearranging terms yields
\begin{multline}
\|\phi_{T,h}^{M}\|^2 + (RePr)^{-1} \Delta t\sum_{n=0}^{M-1}\|\nabla\phi_{T,h}^{n+1}\|^2 
\\
\leq
 C (RePr)^{-1}h^{2k}\||T\||_{2, k+1}^2 + C (RePr)\Delta t\sum_{n=0}^{M-1}\|\nabla\bu_h^{n+1}\|^2 \|\nabla\eta_{T}^{n+1}\|^2
 + C(RePr)h^{2k}\||\nabla T\||_{\infty,0}^2 \||\bu\||_{2, k+1}^2\\
 +C (RePr)\Delta t\sum_{n=0}^{M-1}\|T^{n+1}\|_{H^3}^2\|\bphi_{\bu,h}^n\|^2 + CRePr(\Delta t)^4 (\|T_{ttt}\|_{2,0}^2 + \||\nabla T\||_{\infty, 0}^2 \|\nabla\bu_{tt}\|_{2,0}^2).\label{sumeq2}
\end{multline}
Adding \eqref{sumeq1} and \eqref{sumeq2}, applying Gronwall's Inequality for any $\Delta t>0$ and assuming $(\bP_k,P_{k-1}, P_k)$ or $(\bP_k,P_{k-1}^{disc}, P_k)$ Scott-Vogelius elements gives
\begin{multline}
\|\bphi_{\bu,h}^{M}\|^2 +\|\phi_{T,h}^{M}\|^2
+ \Delta t\sum_{n=0}^{M-1} (Re^{-1}\|\nabla \bphi_{\bu,h}^{n+1}\|^2 +(RePr)^{-1}\|\nabla\phi_{T,h}^{n+1}\|^2) \\
\leq   C( (\Delta t)^4 + \alpha^2h^{2k} + \alpha^2\|a_{D_N}(\bu)\|^2 + h^{2k} + h^{2k+2})\label{gronson}.
\end{multline}
Finally we apply the triangle inequality for error terms
\begin{multline*}
\|\bu(t^*)-\bu_{h}^{M}\|^2 + \|T(t^*)-T_{h}^{M}\|^2 
+ \Delta t\sum_{n=0}^{M-1}(Re^{-1}\|\nabla(\bu^{n+1}-\bu_{h}^{n+1})\|^ 2 + (RePr)^{-1}\|\nabla(T^{n+1}-T_{h}^{n+1})\|^ 2)\\
\leq 2(\|\bphi_{\bu,h}^{M}\|^2 +\|\phi_{T,h}^{M}\|^2 +  \Delta t\sum_{n=0}^{M-1} (Re^{-1}\|\nabla \bphi_{\bu,h}^{n+1}\|^2 +(RePr)^{-1}\|\nabla\phi_{T,h}^{n+1}\|^2)\\
+2(\|\bta_{\bu}^{M}\|^2 +\|\eta_{T}^{M}\|^2 +  \Delta t\sum_{n=0}^{M-1} (Re^{-1}\|\nabla \bta_{\bu}^{n+1}\|^2 +(RePr)^{-1}\|\nabla\eta_{T}^{n+1}\|^2),
\end{multline*}
and from which using the result \eqref{gronson} together with the regularity assumptions completes the proof.
\end{proof}

\section{Numerical experiments}
This section presents two numerical experiments in order to test the theory above and the proposed scheme. The first numerical experiment illustrates the predicted convergence rates using an analytic test problem. The second experiment compares the performance of the proposed algorithm on Marsigli flow with the standard BDF2LE-FEM and the usual Leray-$\alpha$ model of the Boussinesq equations. 
\subsection{Numerical experiment 1: Convergence rate verification.}
The first numerical experiment aims to confirm the spatial and temporal convergence rates of Algorithm~\ref{boualg}. To verify these both rates, we first chose analytical solutions and the dimensionless flow parameters as follows
\begin{align*}
\bu(t, \bx) :=
\begin{bmatrix}
e^t\cos(\pi(y-t)) \\
e^t\sin(\pi(x+t))
\end{bmatrix},\hspace{2mm}
p(t, \bx):&= \sin(x+y)(1+t^2), \hspace{3mm} T(t, \bx):= \sin(\pi x) + ye^t,\\
\nu& = 1,\,\, Ri =1,\,\, \kappa=1,
\end{align*}
and from which we calculate forcing terms of the Boussinesq equations. With the choice of $(\bP_2, P_1, P_2)$ finite elements for the velocity/pressure/temperature solutions and $N=0$, Theorem~\ref{bouscon} concludes the second order temporal convergence, i.e. 
\begin{align*}
\|\bu(t^*)- \bu_h^{M}\| + \||\bu-\bu_h|\|_{2,1} + \|T(t^*) - T_h^{M}\| +\||T - T_h|\|_{2,1} = \mathcal{O}(\Delta t^2),
\end{align*}
where $$\||\cdot|\|_{2,1}:= \bigg( \Delta t \sum_{n=0}^{M-1} \|\nabla(\cdot )\|^2_{L^2} \bigg)^{1/2},$$ To verify this rate, we compute approximations on the domain $\Omega:= (0,1)^2$ with $a_{D_0}$ by setting $\alpha = h$, and $h=1/128$. 
The results have been presented in Table~\ref{BouTemporalRates}, which are consistent with theoretical finding.

Secondly, we confirm that the spatial convergence is of order $h^{k}$, not $h^{k-1/2}$. Therefore, we run Algorithm 3.1 with both $(\bP_2, P_1, P_2)$ and $(\bP_3, P_2, P_3)$ elements, and compute approximate solutions on five successive mesh refinements taking $t^* =0.001$ with time step $\Delta t=0.0001$. We select the dimensionless parameters as in temporal convergence rate verification, and $\alpha=h$, and $N=0$. Table 1 shows our calculated errors and rates; with $(\bP_2, P_1, P_2)$ we observe a rate of 2, and with $(\bP_3, P_2, P_3)$ we observe a rate of 3.
\begin{table}[h!]
\centering
\begin{adjustbox}{width=0.9\textwidth}
\begin{tabular}[h!]{c | ll | ll | ll | ll}
$t^*$ &  $\|\bu(t^*) - \bu_h^M\|$ & Rate & $\|\bu- \bu_h\|_{2,1}$ & Rate & $\|T(t^*) - T_h^M\|$ & Rate & $\||T - T_h\||_{2,1}$ & Rate \\ 
\hline
1/4 	& 5.3654e-2	 & -- & 2.3750e-1  & --   & 2.2163e-3	& --& 8.5334e-3	& --  \\
1/8	 	& 1.4176e-2  & 1.9202  & 5.3349e-2  & 2.1544  & 5.8506e-4 & 1.9215  & 1.9210e-3 & 2.1597 \\
1/16 	& 3.3613e-3  & 2.0764  & 1.1358e-2  & 2.2317  & 1.4098e-4 & 2.0531 & 4.2128e-4 & 2.1890 \\
1/32	& 7.9670e-4  & 2.0769  & 2.5255e-3  & 2.1691  & 3.3935e-5 & 2.0546 & 1.0591e-4 & 1.9919 \\
\end{tabular}
\end{adjustbox}
\captionsetup{width=.9\textwidth, font=scriptsize}
\caption{\label{BouTemporalRates} \small Velocity/temperature temporal errors and rates for $a_{D_0}$ with $(\bP_2, P_1, P_2)$ with a fixed mesh size $h=1/128$ and $\alpha = h$.}
\end{table} \ \\

\begin{table}[h!]
\centering
\begin{adjustbox}{width=0.9\textwidth}
\begin{tabular}[h]{c  | ll | ll | ll | ll}
h &   $\||\bu - \bu_h^{\bP_2}\||_{2,1}$ & Rate  & $\||T - T_h^{P_2}\||_{2,1}$ & Rate & $\||\bu-\bu_h^{\bP_3}\||_{2,1}$ & Rate & $\||T - T_h^{P_3}\||_{2,1}$ & Rate\\ 
\hline
1/4 	 & 2.1087e-3  & --  & 1.4830e-3	& -- & 1.1265e-3 & --- &  7.1432e-5 & ---  \\
1/8	 	&  5.1784e-4 &  2.0258 & 3.6288e-4	& 2.0310 & 1.7261e-5 & 2.7063 & 8.7786e-6 & 3.0245  \\
1/16 	& 1.2705e-4  &  2.0271 & 8.9120e-5  & 2.0257 & 2.5160e-6 & 2.7783  & 1.0913e-6 & 3.0079 \\
1/32	& 3.1533e-5  &  2.0105 & 2.2140e-5	& 2.0091 & 3.3153e-7 & 2.9239  & 1.3623e-7 & 3.0019 \\
1/64	& 7.8666e-6  &  2.0031 & 5.5243e-6	& 2.0028 & 4.2145e-8 & 2.9757  & 1.7024e-8 & 3.0004  \\
\end{tabular}
\end{adjustbox}
\captionsetup{width=.9\textwidth, font=scriptsize}
\caption{\label{BouSpatialRates} \small Velocity/temperature spatial errors and rates found for $a_{D_0}$ with $(\bP_2, P_1, P_2)$ and $(\bP_3, P_2, P_3)$ for a fixed end time $t^*=0.001$ and a time step $\Delta t = 0.0001$.}
\end{table}


\subsection{Numerical experiment 2: Marsigli's experiment}
The second numerical experiment tests the proposed algorithm on a benchmark problem, named Marsigli's experiment. 
For the problem set-up, we follow the paper \cite{JLW03}. Flow region is an insulated box $[0,8]\times[0,1]$ divided at $x=4$. The initial velocity is taken to be zero since the flow is at rest, and the initial temperature on the left hand side of the box is $T_0 = 1.0$, and on the right hand side $T_0=1.5$. The dimensionless flow parameters are set to be $Re=1,000$, $Ri =4, Pr=1$, and the flow starts from rest.

For the direct simulations, we impose homogeneous Dirichlet boundary conditions for the velocity and the adiabatic boundary condition for the temperature, and use $(\bP_2, P_1, P_2)$-velocity-pressure-temperature finite elements. All solutions are computed at $t^*=2, 4, 6, 8,$ taking a time step $\Delta t=0.025$ on a fine mesh, which provides $135,642$ velocity \textit{degrees of freedom} (dof), $17,111$ pressure dof and $67,821$ the temperature dof. The temperature contours and the velocity streamlines of the DNS are presented in Figure 1 and Figure 2. The results indicate that two currents are formed: the upper current moving from left side to the right side and the under current in the opposite direction, and these two currents are separated by a warm/cold interface, along which the strong shear flow and vortex street is formed, which coincides with the physical phenomenon observed by Marsigli. 
\begin{figure}[h!]
\begin{center}
$t^*=2$\\
\vspace{0.05cm}
\includegraphics[width=0.7\textwidth, height=2.5cm]{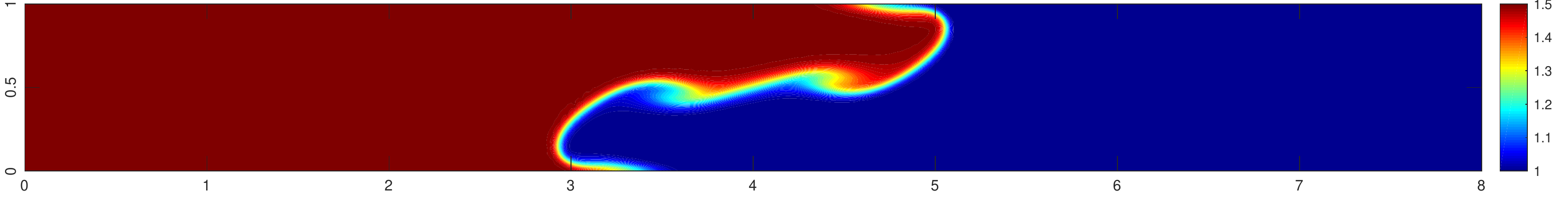}
\\
\vspace{0.2cm}
\includegraphics[width=0.7\textwidth, height=2.5cm]{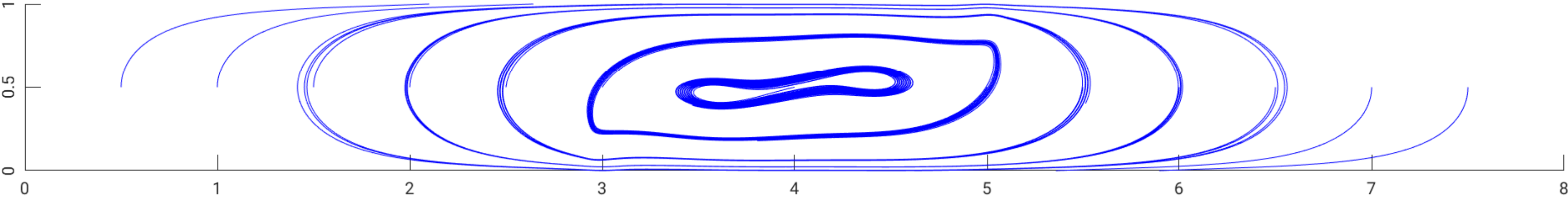}
\\
$t^*=4$\\
\vspace{0.05cm}
\includegraphics[width=0.7\textwidth, height=2.5cm]{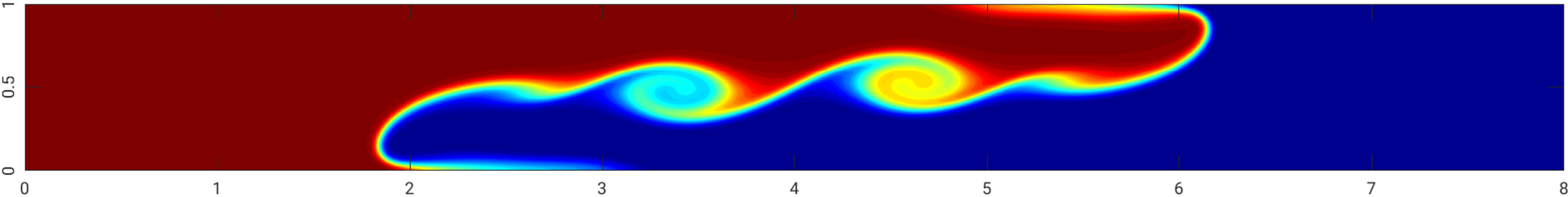}
\\
\vspace{0.2cm}
\includegraphics[width=0.7\textwidth, height=2.5cm]{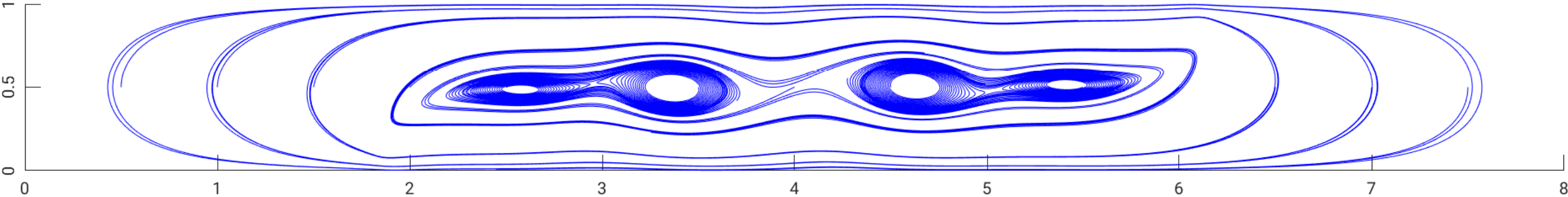}
\\
$t^*=6$\\
\vspace{0.05cm}
\includegraphics[width=0.7\textwidth, height=2.5cm]{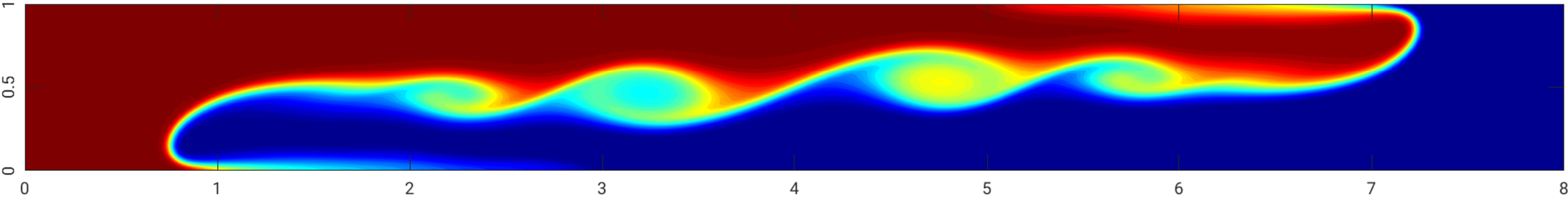}
\\
\vspace{0.2cm}
\includegraphics[width=0.7\textwidth, height=2.5cm]{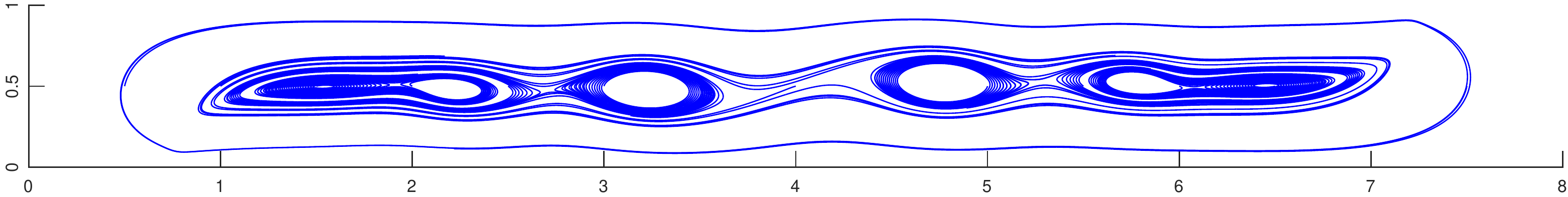}
\\
$t^*=8$\\
\vspace{0.05cm}
\includegraphics[width=0.7\textwidth, height=2.5cm]{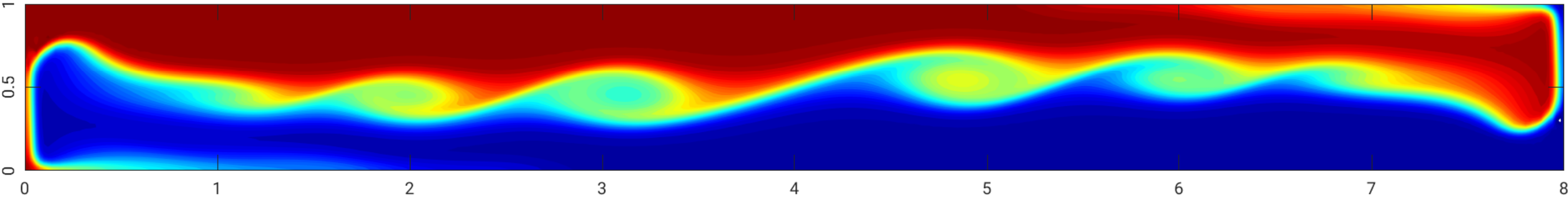}
\\
\vspace{0.2cm}
\includegraphics[width=0.7\textwidth, height=2.5cm]{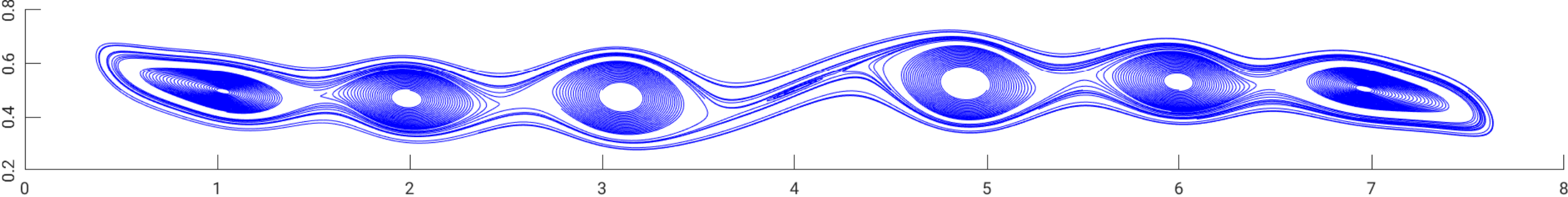}
\end{center}
\caption{The temperature contours and velocity streamlines of fine mesh simulations for 2D Marsigli's flow at $t^*=2, 4, 6, 8$.}
\label{Coarse1}
\end{figure}

Coarse mesh computations were made for the Boussinesq equations (i.e. no model), Leray-$\alpha$ (i.e. $a_{D_N}=1$), and Leray nonlinear filtering with indicator functions $a_{D_0}$ and $a_{D_1}$. Taking the same flow parameters as in DNS except time step $\Delta t=0.02$, we computed and compared all solutions at $t^*=2, 4, 8$ on the same mesh, which gives $26,082$ velocity dof, $3,321$ pressure dof and $13,041$ temperature dof. The results from these computations are shown in Figure \ref{Coarse1}-Figure \ref{Coarse3}. It can be clearly seen that the Algorithm~\ref{boualg} catches very well the flow pattern and temperature distribution of the DNS at each time level. However, BDF2LE-FEM and Leray-$\alpha$ model creates very poor solutions, and builds significant oscillations in temperature and velocity as time progresses. In fact, for larger times, these two methods produce temperature contours which have no physical meaning. This is because it predicts temperatures almost entirely out of the interval $[1.0, 1.5]$.
\begin{figure}[h!]
\begin{center}
$t^*=2$\\
\vspace{0.05cm}
\includegraphics[width=0.89\textwidth, height=2.5cm]{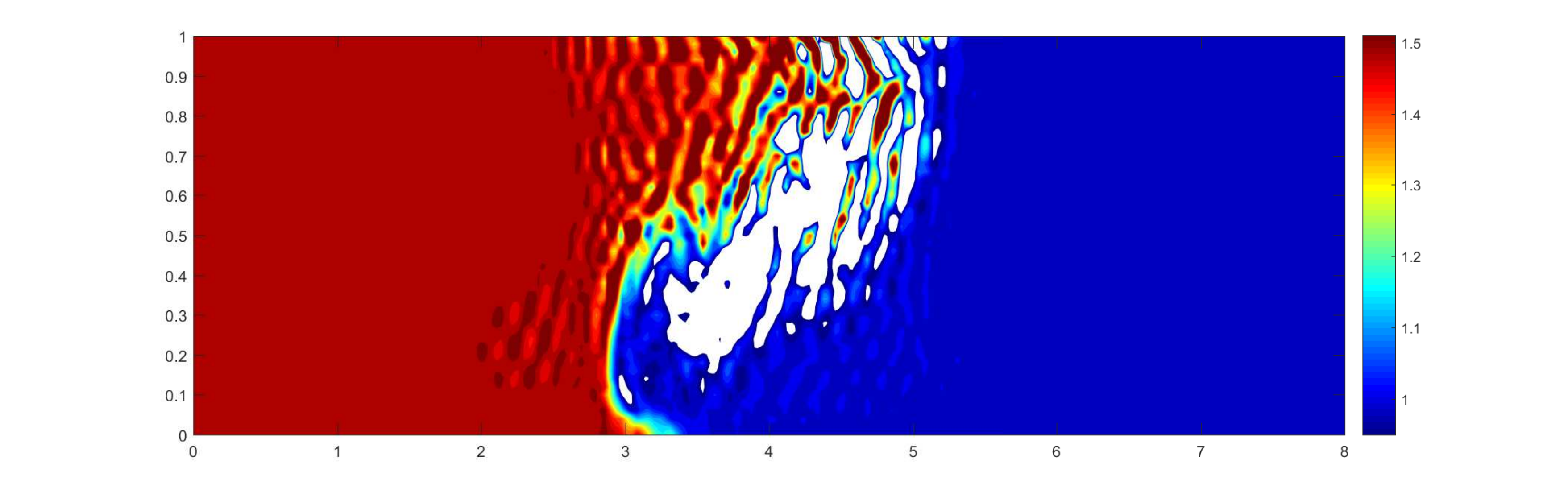}
\\
\vspace{0.2cm}
\includegraphics[width=0.89\textwidth, height=2.5cm]{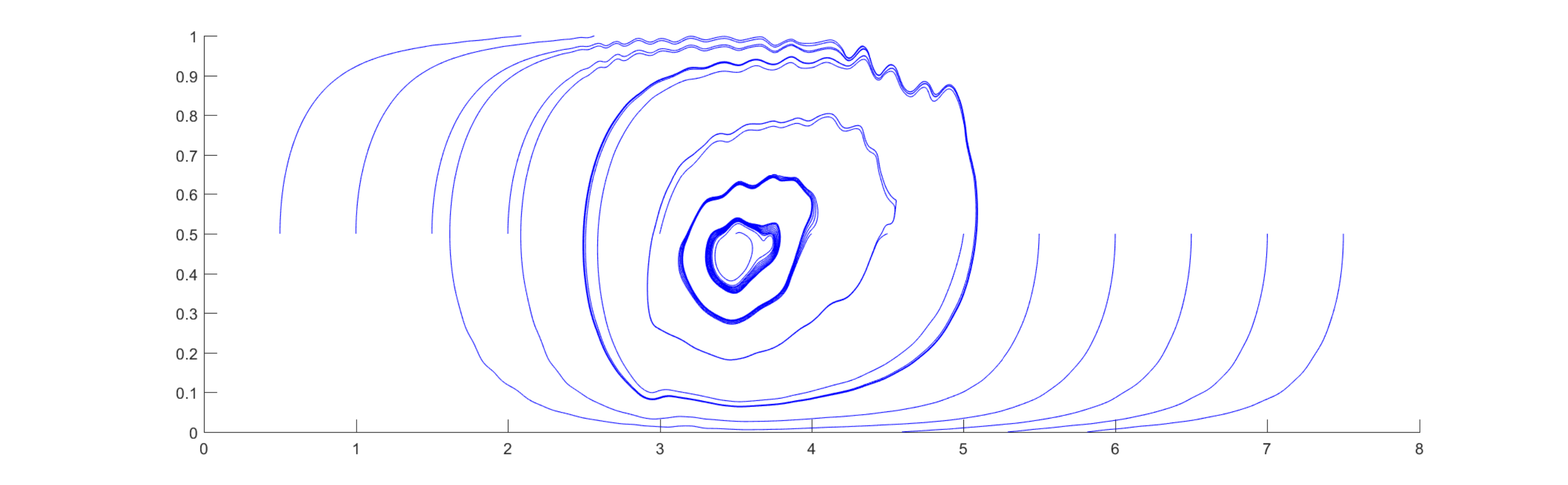}
\\
\vspace{0.2cm}
\includegraphics[width=0.89\textwidth, height=2.5cm]{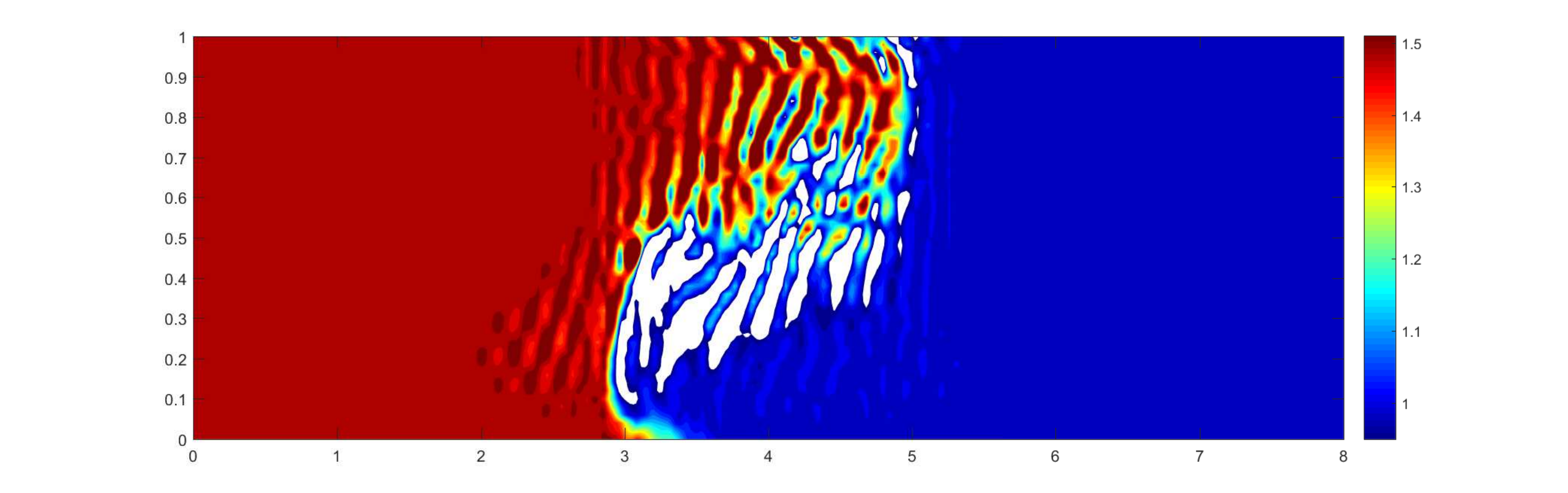}
\\
\vspace{0.2cm}
\includegraphics[width=0.89\textwidth, height=2.5cm]{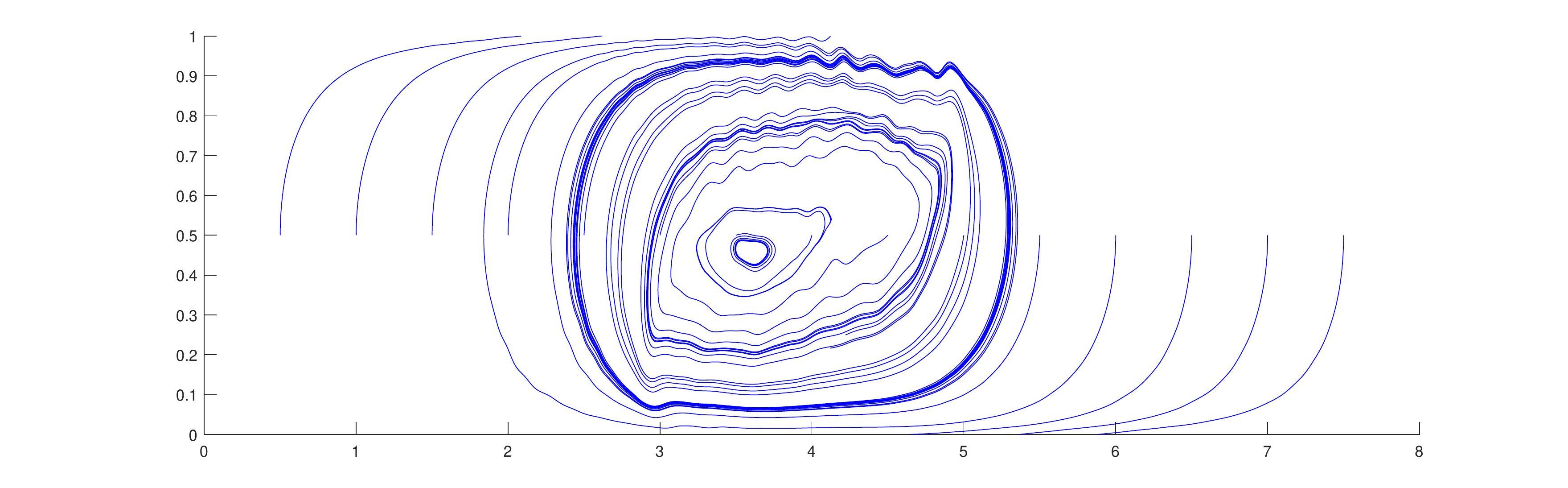}
\\
\vspace{0.2cm}
\includegraphics[width=0.89\textwidth, height=2.5cm]{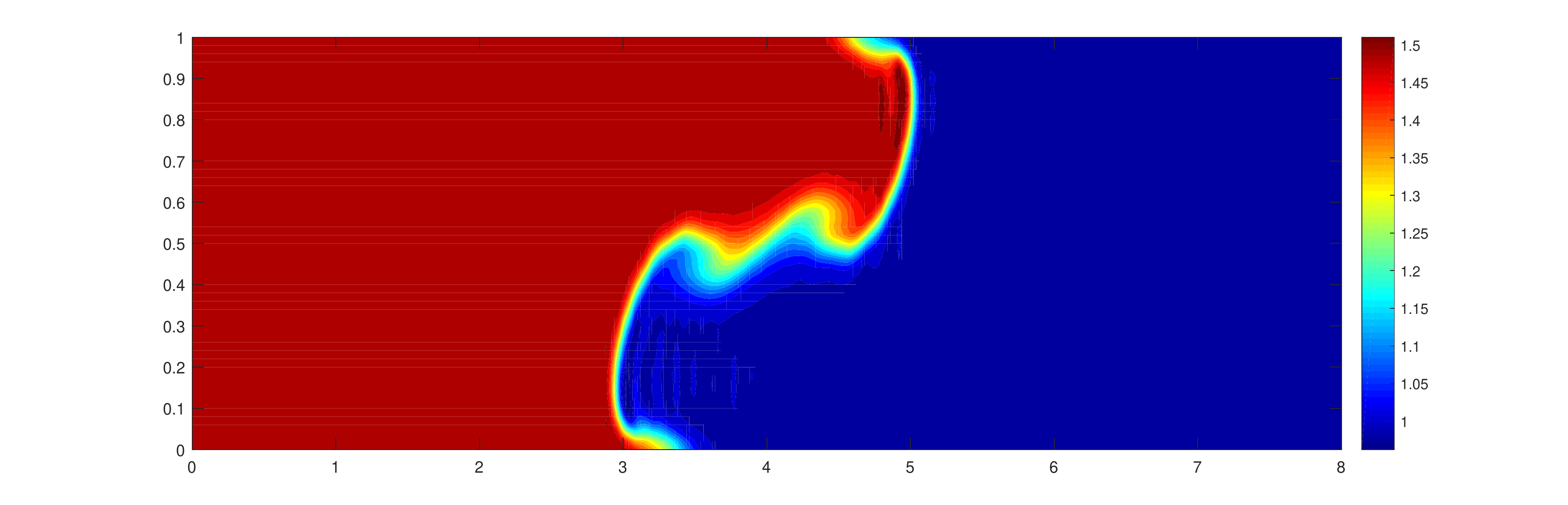}
\\
\vspace{0.2cm}
\includegraphics[width=0.89\textwidth, height=2.5cm]{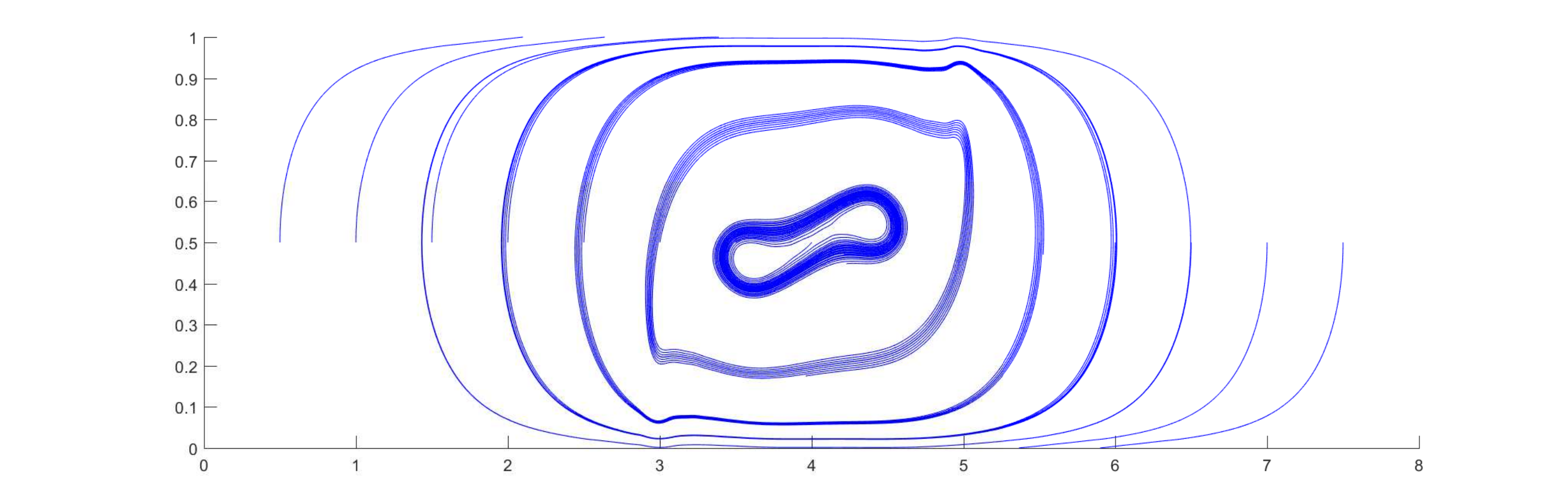}
\\
\vspace{0.2cm}
\includegraphics[width=0.89\textwidth, height=2.5cm]{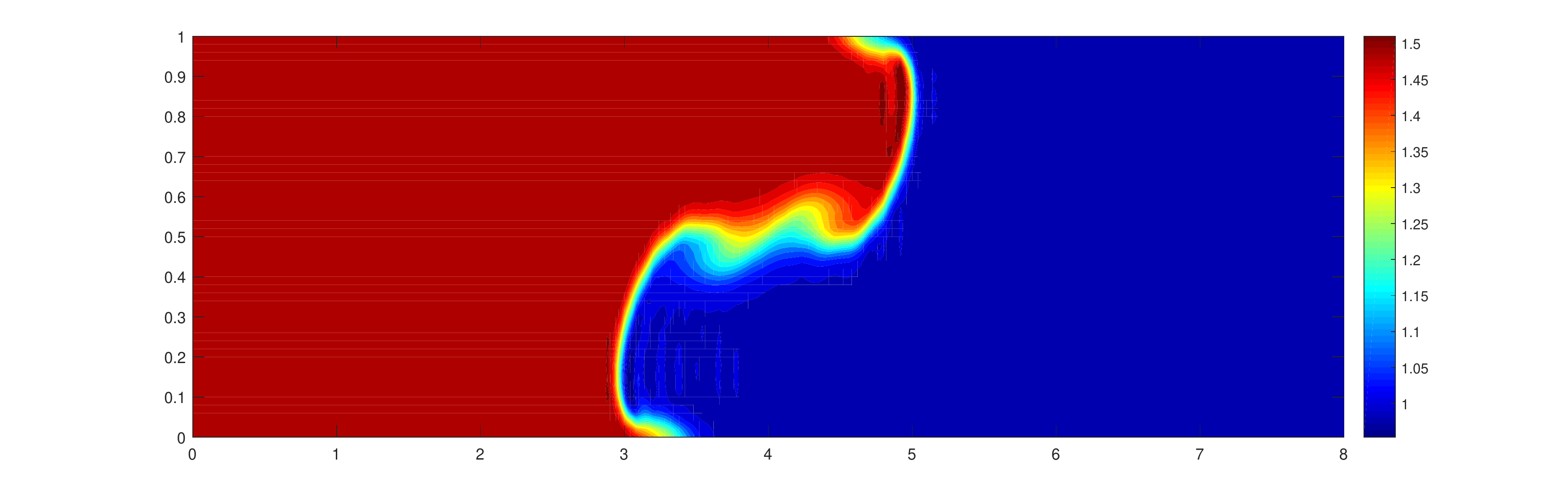}
\\
\vspace{0.2cm}
\end{center}
\caption{The temperature contours and velocity streamlines of coarse mesh simulations for 2D Marsigli's flow, from top to bottom, for Boussinesq (no model), Leray-$\alpha$, and Leray with nonlinear filter that used indicator functions $a_{D_0}$ and $a_{D_1}$.}
\label{Coarse1}
\end{figure}

\begin{figure}[h!]
\begin{center}
$t^*=4$\\
\vspace{0.05cm}
\includegraphics[width=0.89\textwidth, height=2.5cm]{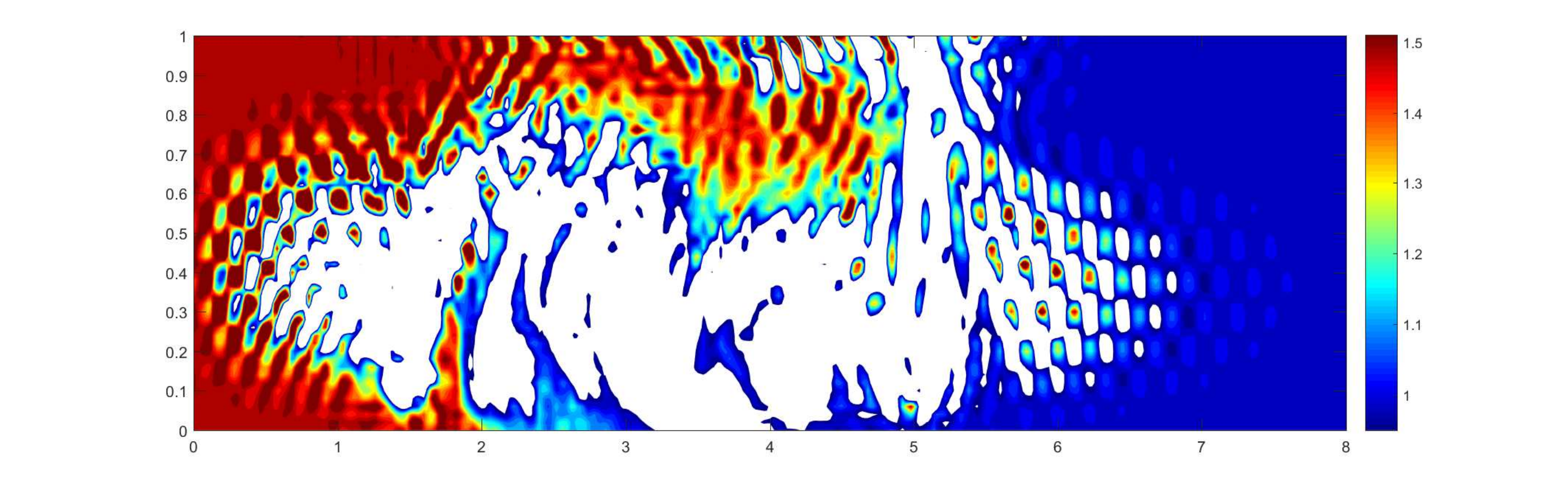}
\\
\vspace{0.2cm}
\includegraphics[width=0.89\textwidth, height=2.5cm]{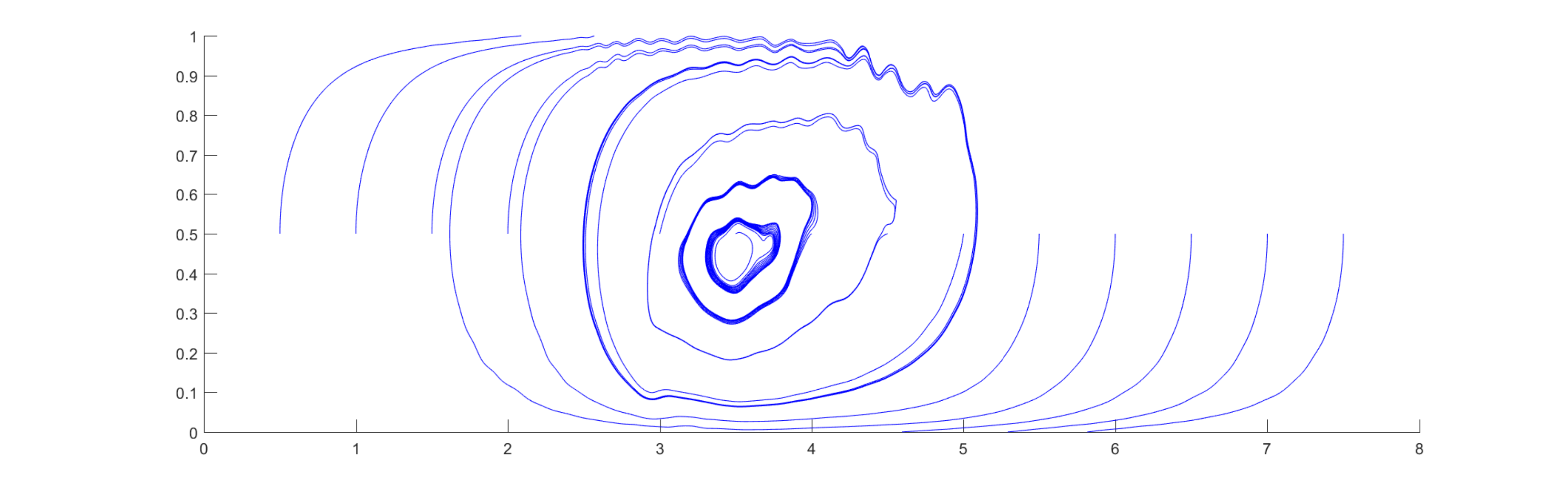}
\vspace{0.2cm}
\includegraphics[width=0.89\textwidth, height=2.5cm]{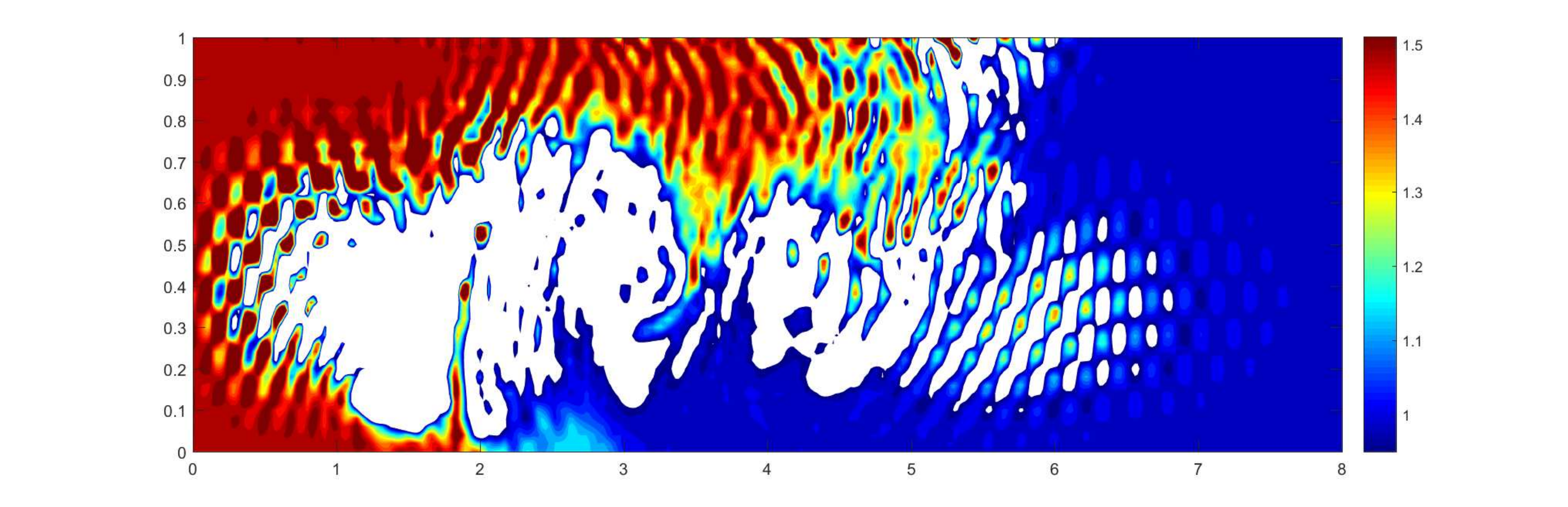}
\\
\vspace{0.2cm}
\includegraphics[width=0.89\textwidth, height=2.5cm]{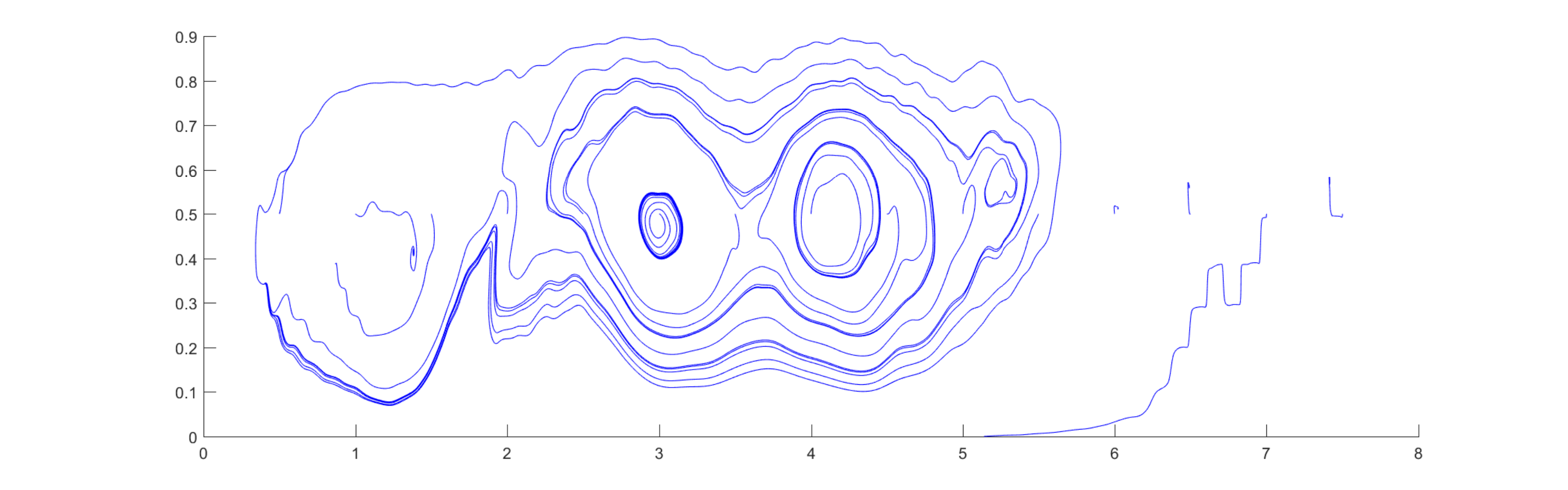}
\\
\vspace{0.2cm}
\includegraphics[width=0.89\textwidth, height=2.5cm]{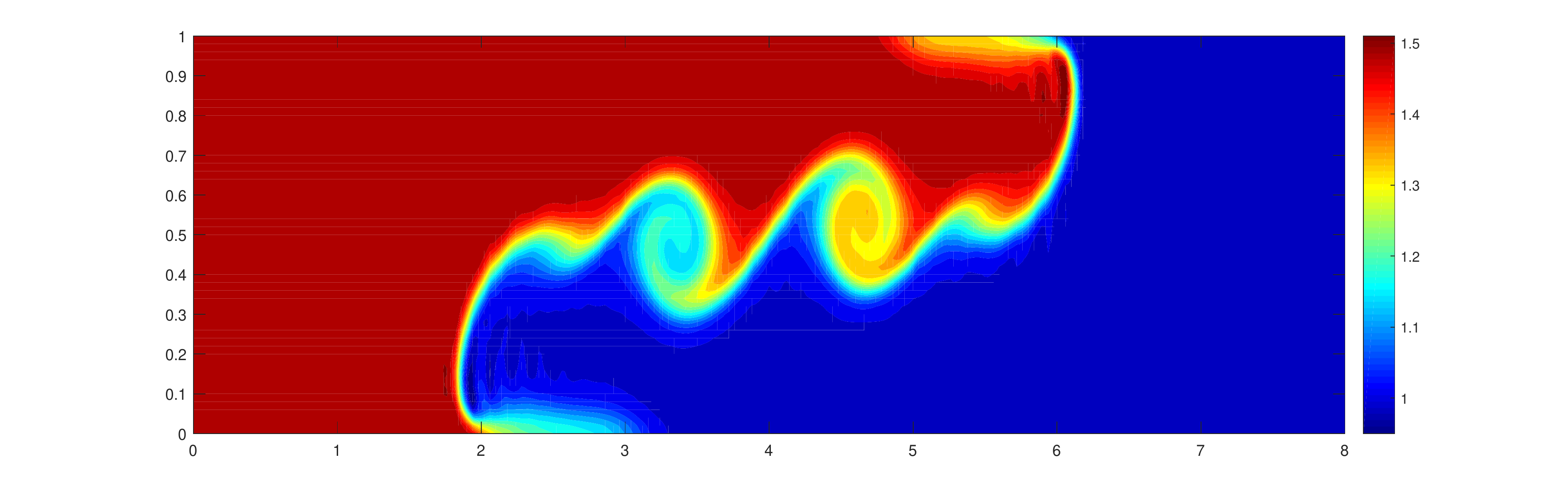}
\\
\vspace{0.2cm}
\includegraphics[width=0.89\textwidth, height=2.5cm]{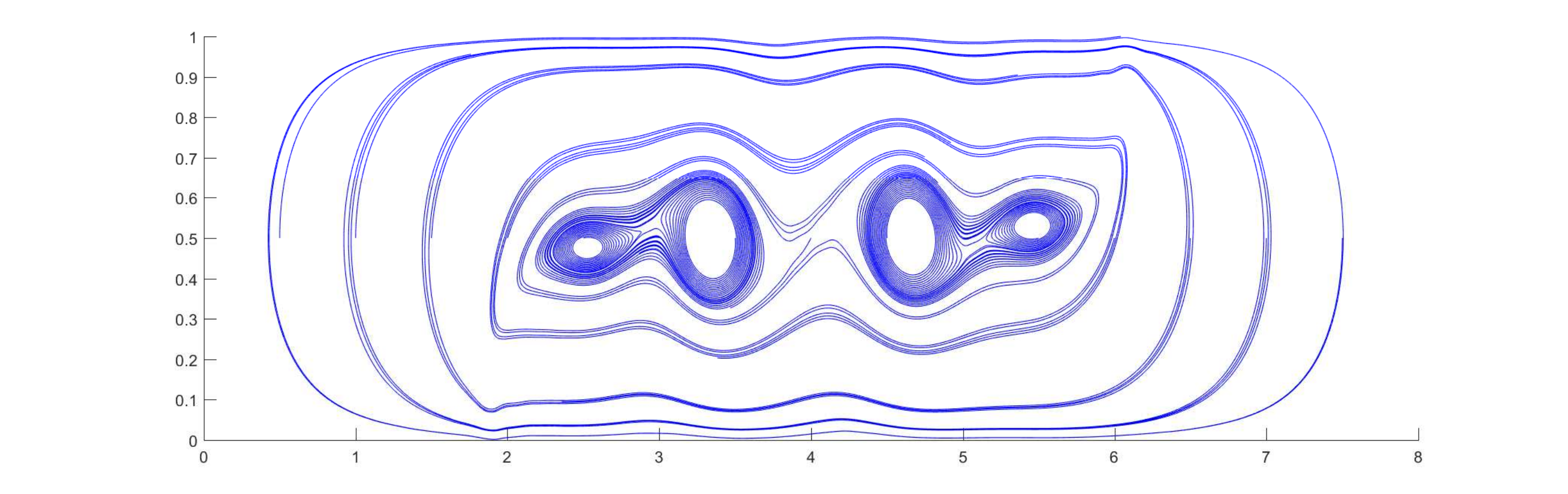}
\\
\vspace{0.2cm}
\includegraphics[width=0.89\textwidth, height=2.5cm]{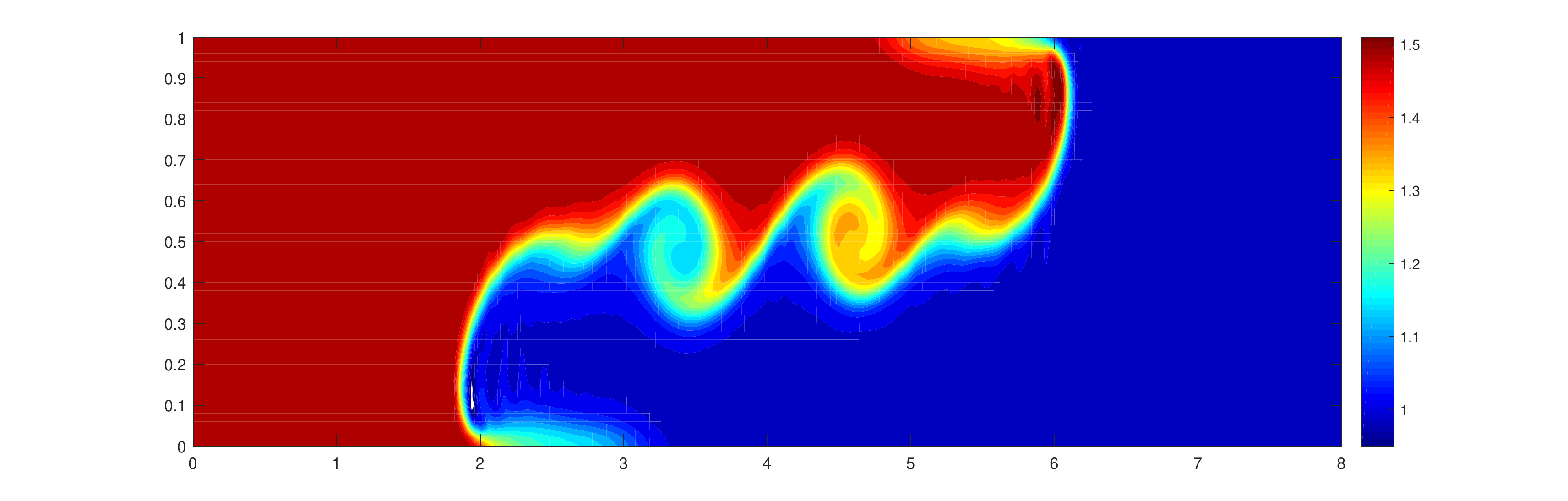}
\\
\vspace{0.2cm}
\includegraphics[width=0.89\textwidth, height=2.5cm]{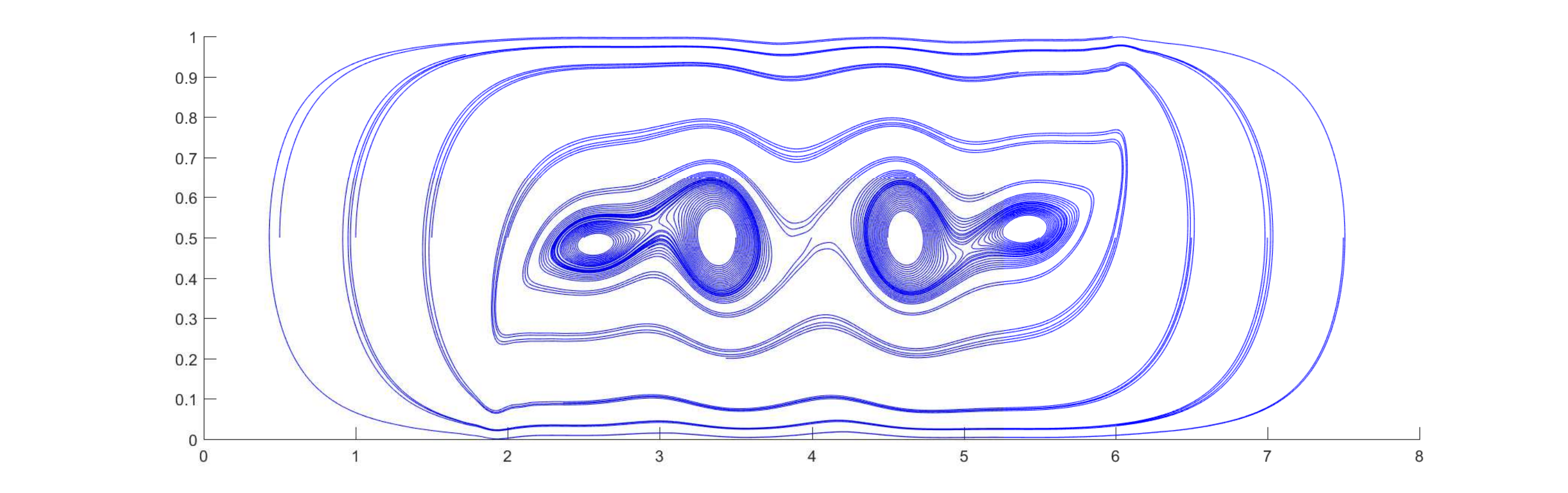}
\end{center}
\caption{The temperature contours and velocity streamlines of coarse mesh simulations for 2D Marsigli's flow, from top to bottom, for Boussinesq (no model), Leray-$\alpha$, and Leray with nonlinear filter that used indicator functions $a_{D_0}$ and $a_{D_1}$.}
\label{Coarse2}
\end{figure}

\begin{figure}[h!]
\begin{center}
$t^*=8$\\
\vspace{0.05cm}
\includegraphics[width=0.89\textwidth, height=2.5cm]{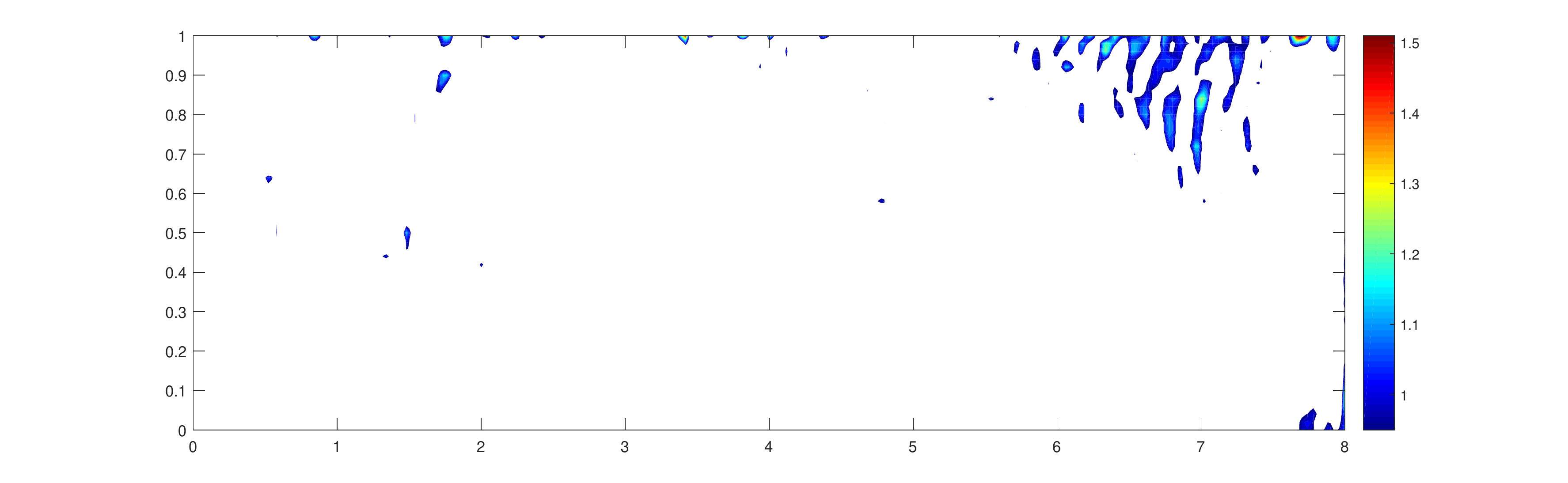}
\\
\vspace{0.2cm}
\includegraphics[width=0.89\textwidth, height=2.5cm]{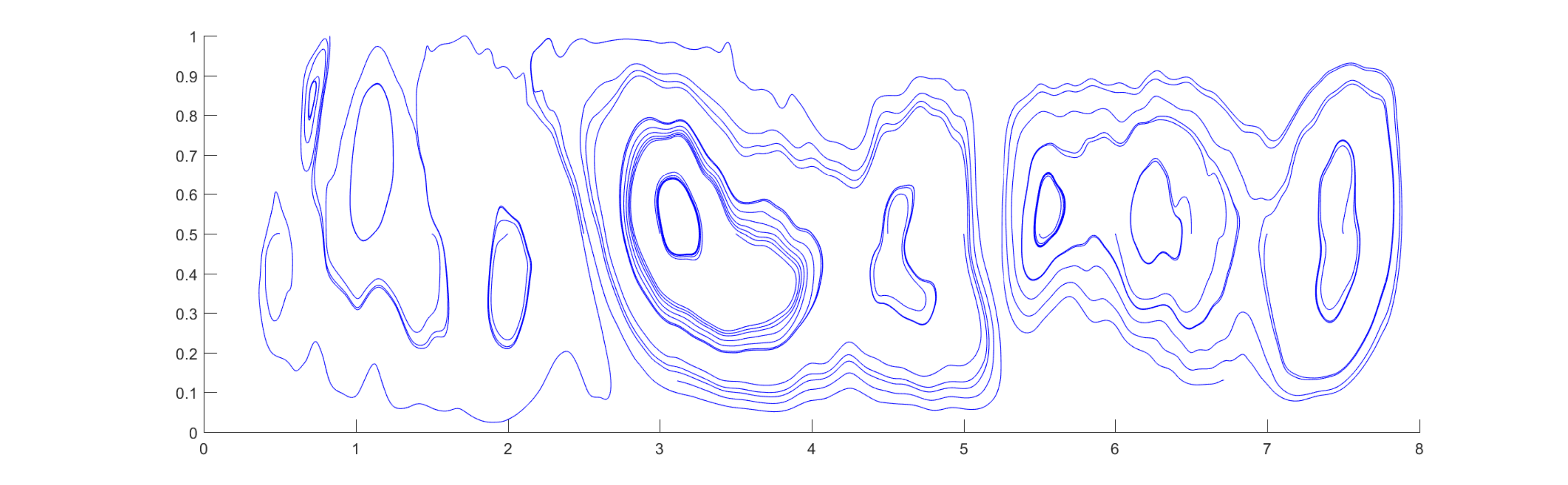}
\\
\vspace{0.2cm}
\includegraphics[width=0.89\textwidth, height=2.5cm]{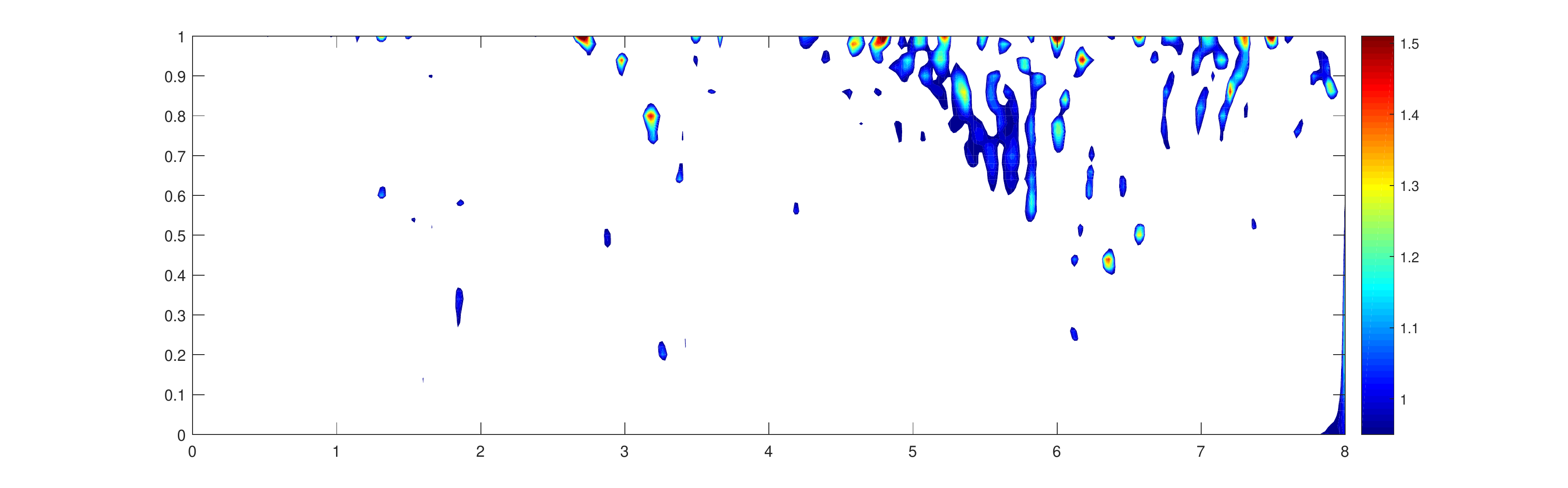}
\\
\vspace{0.2cm}
\includegraphics[width=0.89\textwidth, height=2.5cm]{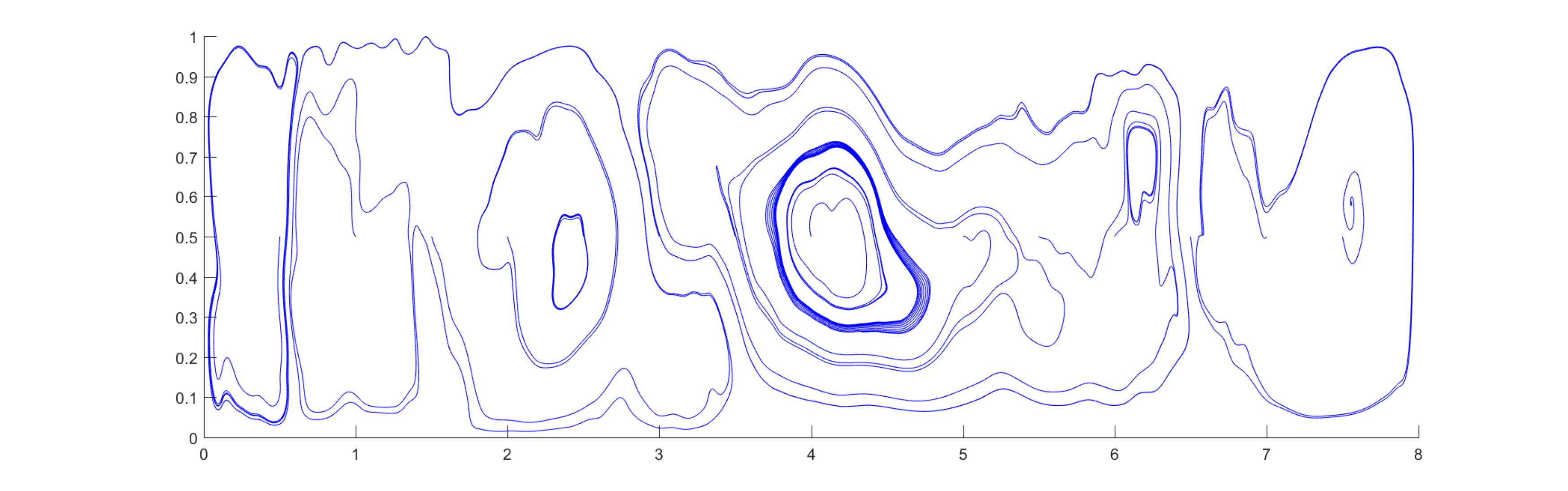}
\\
\vspace{0.2cm}
\includegraphics[width=0.89\textwidth, height=2.5cm]{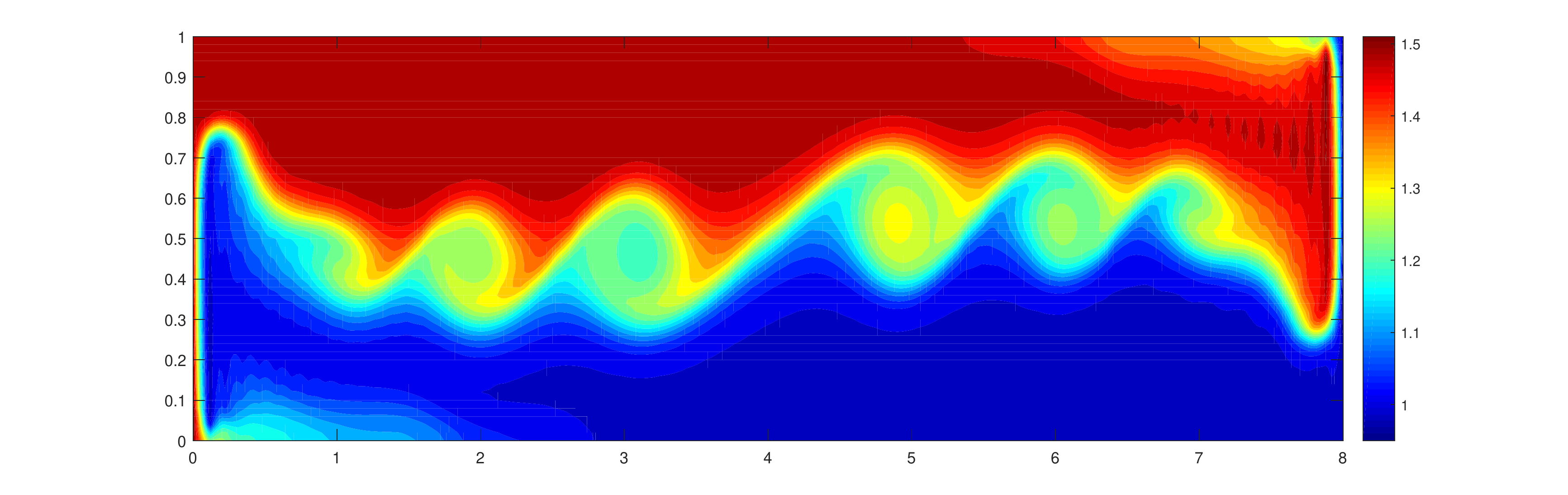}
\\
\vspace{0.2cm}
\includegraphics[width=0.89\textwidth, height=2.5cm]{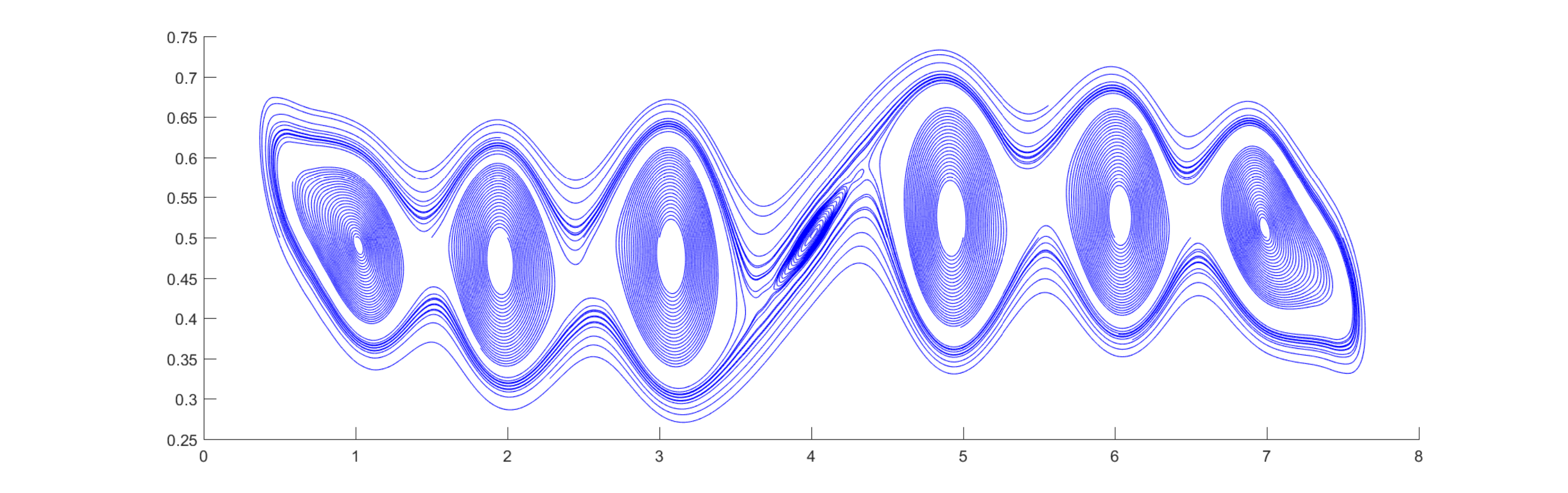}
\\
\vspace{0.2cm}
\includegraphics[width=0.89\textwidth, height=2.5cm]{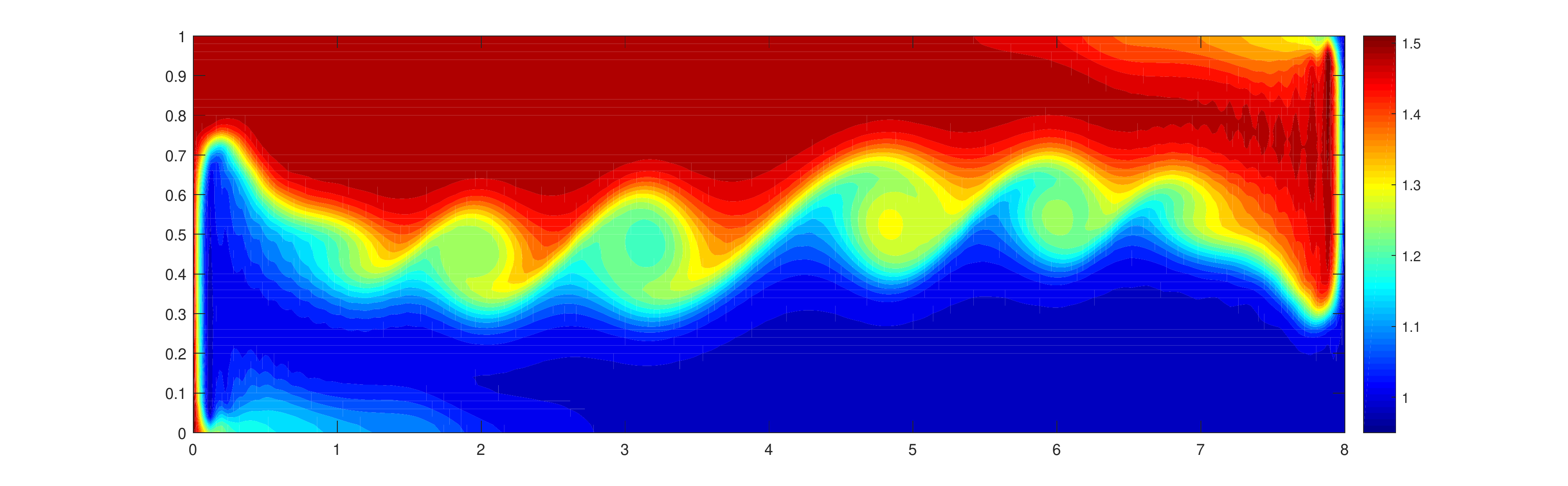}
\\
\vspace{0.2cm}
\includegraphics[width=0.89\textwidth, height=2.5cm]{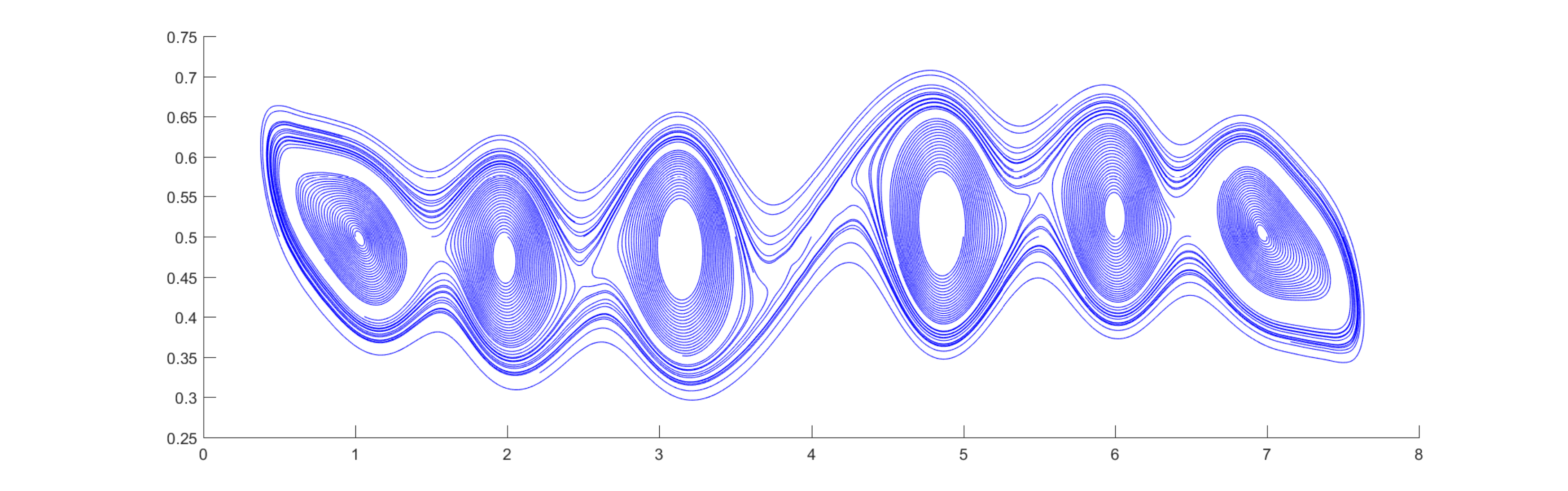}
\end{center}
\caption{The temperature contours and velocity streamlines of coarse mesh simulations for 2D Marsigli's flow, from top to bottom, for Boussinesq (no model), Leray-$\alpha$, and Leray with nonlinear filter that used indicator functions $a_{D_0}$ and $a_{D_1}$.}
\label{Coarse3}
\end{figure}
\section{Conclusions}
In this paper, we studied the Leray regularization model with adaptive non-linear filtering of incompressible, non-isothermal fluid flows. Indicator functions enable us to choose the filtering radius locally so that one can determine the regions where a flow simulation needs a regularization. The numerical method we proposed for the model used BDF2-FE discretization with the linear extrapolation of filtered velocity term. The filter step was also linearized. Hence, the velocity-pressure-temperature system was solved with the discrete velocity filtering simultaneously at each time level. In the implementation, this leads to negligible extra cost, resulted from the calculation of the $a(\bu)$, when compared to the usual Leray-$\alpha$ model. We analyzed the scheme rigorously; proved unconditional stability and the convergence of the scheme. We verified the optimal convergence rates, and tested the algorithm to show its effectiveness on Marsigli's experiment. We observed that our method gives much more accurate solutions on a coarser mesh when compared to the underresolved direct numerical simulation and the Leray-$\alpha$ model for the Boussinesq system. \ \\ \\

\noindent \textbf{Conflict of Interest:} The authors declare no conflict of interest.

\end{document}